\documentclass[12pt,leqno,a4paper,article]{amsart}
\usepackage{amsmath, amsthm, verbatim, amsfonts, amssymb, color}
\usepackage{mathrsfs}
\usepackage{dsfont}
\usepackage[a4paper,margin=3cm]{geometry}

\usepackage{graphicx}

\usepackage{algorithm,algorithmic}
\usepackage{booktabs}
\usepackage{multirow}
\usepackage{bigstrut}

\usepackage{float}

\usepackage{bm}
\usepackage{bbm}

\usepackage{ulem}
\usepackage{url}

\usepackage{mathtools}
\usepackage{mathrsfs}
\usepackage{bm}
\usepackage{marginnote}
\usepackage{extarrows}
\usepackage{xcolor}

\usepackage{enumitem}
\usepackage{hyperref}
\hypersetup{hypertex=true,
	colorlinks=true,
	linkcolor=blue,
	anchorcolor=blue,
	urlcolor=cyan,
	citecolor=green,
	pdftitle={Overleaf Example},
	pdfpagemode=FullScreen}

\usepackage{subfigure}

\makeatletter
\newcommand{\runum}[1]{\mathrm{\romannumeral #1}}
\newcommand{\Rmnum}[1]{\mathrm{\expandafter\@slowromancap\romannumeral #1@}}
\makeatother

\newcommand\zero{\mathbf{0}}

\newcommand\dd{\,\mathrm{d}}
\newcommand{\rank}{\operatorname{rank}}

\newcommand{\T}{\operatorname{T}}

\newcommand{\floor}[1]{\left\lfloor #1 \right\rfloor}

\newcommand{\abs}[1]{\left|#1\right|}
\newcommand*\norm[1]{\left\lVert#1\right\rVert}

\newcommand{\opnorm}[1]{\left\Vert #1 \right\Vert_{\operatorname{op}}}

\newcommand{\inprod}[1]{\left\langle #1 \right\rangle}

\newcommand{\argmin}{\mathop{\arg\min}}

\newcommand{\bbP}{\mathbb{P}}
\newcommand{\bbR}{\mathbb{R}}
\newcommand{\bbE}{\mathbb{E}}

\newcommand{\bfx}{\mathbf{x}}
\newcommand{\bfy}{\mathbf{y}}

\newcommand{\bfA}{\mathbf{A}}
\newcommand{\bfB}{\mathbf{B}}

\newcommand{\bfI}{\mathbf{I}}

\newcommand{\bfX}{\mathbf{X}}
\newcommand{\bfY}{\mathbf{Y}}
\newcommand{\bfZ}{\mathbf{Z}}

\newcommand{\bmPhi}{\bm{\Phi}}
\newcommand{\bmPsi}{\bm{\Psi}}

\newcommand{\bmepsilon}{\bm{\varepsilon}}
\newcommand{\bmzeta}{\bm{\zeta}}
\newcommand{\bmtheta}{\bm{\theta}}

\newcommand{\rme}{\mathrm{e}}
\newcommand{\rmd}{\mathrm{d}}

\newcommand{\diag}{\mathrm{diag}}
\newcommand{\popR}{R_{\ell_1}} 
\newcommand{\empR}{\widehat{R}_{\ell_1}} 
\newcommand{\catoniR}{ \widehat{R}_{\psi_{\alpha}, \ell_1} } 

\newcommand{\pheq}{\mathrel{\phantom{=}}}

\theoremstyle{plain}
\newtheorem{theorem}{Theorem}[section]

\newtheorem{lemma}[theorem]{Lemma}

\newtheorem{assumption}[theorem]{Assumption}

\theoremstyle{definition}
\newtheorem{remark}{Remark}[section]

\newtheorem{definition}[theorem]{Definition}

\numberwithin{equation}{section}
\renewcommand\labelenumi{\textup{\alph{enumi})}}
\renewcommand\theenumi\labelenumi

\usepackage{url}

\usepackage{marginnote}

\begin{document}
	
	\title[Robust estimation ]{ Robust estimation for high-dimensional time series with heavy tails}
	
	\allowdisplaybreaks[4]
	
	\date{\today}
	
	\author[Y. Wang]{Yu Wang}
	\address{Yu Wang: 1. Department of Mathematics,
		Faculty of Science and Technology,
		University of Macau,
		Av. Padre Tom\'{a}s Pereira, Taipa
		Macau, China; \ \ 2. UM Zhuhai Research Institute, Zhuhai, China.}
	\email{yc17447@um.edu.mo}
	
	\author[G. Li]{Guodong Li}
	\address{Guodong Li: Department of Statistics \& Actuarial Science,
		University of Hong Kong,
		Pokfulam Road, Hong Kong, China.}
	\email{gdli@hku.hk}
	
	\author[Z Xiao]{Zhijie Xiao}
	\address{Zhijie Xiao: Department of Economics,
		Boston College,
		Chestnut Hill, MA 02467 USA.}
	\email{zhijie.xiao@bc.edu}
	
	\author[L. Xu]{Lihu Xu}
	\address{Lihu Xu: 1. Department of Mathematics,
		Faculty of Science and Technology,
		University of Macau,
		Av. Padre Tom\'{a}s Pereira, Taipa
		Macau, China; \ \ 2. UM Zhuhai Research Institute, Zhuhai, China.}
	\email{lihuxu@um.edu.mo}
	
	\author[W. Zhang]{Wenyang Zhang}
	\address{Wenyang Zhang:  Faculty of Business Administration,
		University of Macau,
		Av. Padre Tom\'{a}s Pereira, Taipa,
		Macau, China.}
	\email{wenyangzhang@um.edu.mo}
	
	
	
	\maketitle
	
	\begin{abstract}
		
	We study in this paper the problem of least absolute deviation (LAD) regression for high-dimensional heavy-tailed time series which have finite $\alpha$-th moment with $\alpha \in (1,2]$. To handle the heavy-tailed dependent data, we propose a Catoni type truncated minimization problem framework and obtain an $\mathcal{O}\big( \big( (d_1+d_2) (d_1\land d_2) \log^2 n / n \big)^{(\alpha - 1)/\alpha} \big)$ order excess risk, where $d_1$ and $d_2$ are the dimensionality and $n$ is the number of samples. We apply our result to study the LAD regression on high-dimensional heavy-tailed vector autoregressive (VAR) process. Simulations for the VAR($p$) model show that our new estimator with truncation are essential because the risk of the classical LAD has a tendency to blow up.  We further apply our estimation to the real data and find that ours fits the data better than the classical LAD. 
		
	\end{abstract}
	
	\tableofcontents
	
	\noindent
	
	\section{Introduction}
	\label{sec:sec1}
	The least absolute deviation (LAD) regression (see, e.g. \cite{d6e7bc40-f957-3e2f-8705-9442b9810a61, leastKani, MR0514166, MR1128411}) is a classical statistical model aiming to find the unknown minimizier $\bmtheta^*$ of the following minimization problem:
	\begin{equation}
		\label{e:PopMin1}
		\bmtheta^* = \argmin_{\bmtheta \in \bm{\Theta} } R_{\ell_1}(\bmtheta)
		\quad 
		\text{with} \quad 
		\popR(\bmtheta) = \bbE_{(\bfX, Y) \sim \bm{\Pi}} \big[ \abs{Y- \inprod{\bfX, \bmtheta}} \big],
	\end{equation}
	where $\bm{\Pi}$ is the population's distribution, and $\bm{\Theta} \subseteq \bbR^d$ is the set in which $\bmtheta^*$ is located. Because $\bm{\Pi}$ is not known in practice, one usually draws $n$ i.i.d. samples $\{ (\bfX_i , Y_i) \in \bbR^d \times \bbR, i=1, \dotsc n \}$ from $\bm{\Pi}$, and considers the following empirical optimization problem:
	\begin{equation}
		\label{e:EmpMin-0}
		\bar{\bmtheta} = \argmin_{\bmtheta \in \bm{\Theta} } \empR(\bmtheta)
		\quad \text{with} \quad
		\empR(\bmtheta) = \frac{1}{n} \sum_{i=1}^{n} \abs{ Y_i - \inprod{\bfX_i, \bmtheta} }.
	\end{equation}
	In order to obtain a theoretical guarantee for the model \eqref{e:EmpMin-0}, most of the known papers have the assumption that the distribution $\bm{\Pi}$ is {bounded} or {sub-Gaussian}; see for instance \cite{Lucien1998, Koltchinskii2011OracleII, Zhang2017EmpiricalRM}.
	
	\cite{MR2906886} proposed a novel robust ridge least square regression for the data with finite $4$-th moment, and proved that their estimation has an excess risk in the order $\mathcal{O}(d/n)$. Since the LAD has a strong robustness, it is natural to ask whether one can extend this result to the LAD with finite $2$-nd moment and obtain an excess risk in the order
	$\mathcal{O}(\sqrt{d/n})$; see Zhang et al. \cite{NEURIPS2018_8b16ebc0} for more details. It seems that this question has not been answered. A partial result was by Zhang et al. \cite{NEURIPS2018_8b16ebc0}, 
	they introduced a Catoni type truncation function $\psi_2(r): \bbR \to \bbR$ satisfying 
	\begin{equation} \label{e:C-F2}
		-\log\left(1 - r + \frac{r^2}2\right)  \le  \psi_2(r)  \le  \log\left(1 + r + \frac{r^2}2\right),
	\end{equation}
	and considered a truncated LAD problem as the following: 
	\begin{equation}
		\label{e:C-Min1}
		\hat{\bmtheta} = \argmin_{\bmtheta \in \bm{\Theta} } \widehat R_{\psi_2, \ell_1} (\bmtheta)
		\quad \text{with} \quad
		\widehat R_{\psi_2, \ell_1}(\bmtheta) = \frac{1}{n \lambda} \sum_{i=1}^n \psi_2( \lambda \abs{Y_i - \inprod{\bfX_i, \bmtheta}} ),
	\end{equation}
	where $\lambda > 0$ is a tuning parameter. Catoni's innovative truncation idea \cite{MR3052407} has been broadly applied to lots of research problems in statistics, computer science, econometrics and so on; see for instance \cite{MR3124669, MR3819104, MR4112627, MR3983786, MR4017683, MR4078461, NIPS2017_10c272d0, pmlr-v115-xu20b,MR4352550,MR4582514}. 
	
	In this paper, we shall use a simple $1$-dimensional toy model to show that it is impossible to extend the robust least square regression in \cite{MR2906886} to the LAD with finite $2$-nd moment and obtain an excess risk in the order $\sqrt{d/n}$. So a truncation is essential to mitigate the bad effect of outliers; see Section \ref{ss:NecTru}.
	
	Unfortunately, the method \eqref{e:C-Min1} does not work any more when the data only have finite $\alpha$-th moment with $\alpha \in (1,2)$. It is well known that most of data in finance and networks are heavy-tailed and not independent. For instance,  Pareto law \cite{MR4289888, MR4081550, MR3910009} describes the distributions of wealth and social networks, while
	Fr\'echet (inverse Weibull) law \cite{invweibull2008, generalinvwei2011} is used to model failure rates in reliability and biological studies. The results in \cite{MR2906886} and \cite{NEURIPS2018_8b16ebc0} are both established under the assumption that the observed data are i.i.d., so this methods do not work for time series.  
	Another disadvantage of the models \eqref{e:EmpMin-0} and \eqref{e:C-Min1} is that they can not be fitted into the multidimensional time series because the output $Y$
	is only one dimensional. See more details in Section \ref{sec:AR}.  
	
	Motivated by solving the regression problem for high-dimensional time series only having $\alpha$-th moment with $\alpha \in (1,2]$ and fitting our theory into the high-dimensional time series model such as vector autoregressoin (VAR), we will propose a truncation function 
	$\psi_\alpha(r)$ and develop a new robust estimation for the time series data $\{ (\bfX_i, \bfY_i) \in \bbR^{d_1} \times \bbR^{d_2}, i\in \mathbb{N} \}$; see more details in \eqref{e:PopLasso} and \eqref{e:EmpMin}.  
	{\bf Our contributions} are summarized as the below: 
	
	(i). We establish a truncated LAD estimation framework for the high-dimensional time series which only have $\alpha$-th moment with $\alpha \in (1,2]$. The idea is to replace the truncation function $\psi_2$ with the one $\psi_\alpha$ defined by \eqref{e:T-func} and add an $\ell_1$ penalty. On the one hand,
	as mentioned before, because the model with the truncation function $\psi_2$ only works for the regression problems with finite variance, our modification of $\psi_\alpha$ is essential and solves the moment problem in the case $\alpha \in (1,2)$. On the other hand, our result is for the high-dimensional time series rather than the i.i.d. data in the previous works. We add an $\ell_1$ penalty to solve the high-dimensional problem, and use block technique to handle the dependence problem. In our simulation, we find that the $\ell_1$ penalty is crucial for the convergence of stochastic gradient descent (SGD) algorithms. 
	
	(ii). Since the output $Y$ in the models \eqref{e:EmpMin-0} and \eqref{e:C-Min1} is one dimensional, we can not apply this model to study multivariate time series such as the classical VAR model. In order to overcome this problem, we assume that both the input $\bfX \in \bbR^{d_1}$ and the output $\bfY \in \bbR^{d_2}$ are high-dimensional and the parameter $\bm \theta \in \bbR^{d_1 \times d_2}$ is a high-dimensional matrix. The consequent landscape of the loss function $\widehat R_{\psi_\alpha, \ell_1}(\bmtheta)$ in the regression \eqref{e:C-min} is highly complex. The excess risk, commonly used in the data science \cite{MR3837109} and high-dimensional statistical inference \cite{MR2906886}, is a natural measurement for the quality of the estimator $\hat{\bmtheta}$. We show in Theorem \ref{cor:M_n = m_n = log n} below that the excess risk of the regression problem \eqref{e:C-min} is bounded by $\mathcal{O}\big( \big( { \log n  } ( \abs{\log \varepsilon} + (d_1+d_2) (d_1 \land d_2) \log n )/ {n}\big)^{{(\alpha-1)} / {\alpha}} \big)$, where $\epsilon \in (0,1)$ is a small error tolerance. Even in the classical LAD case, such an extension has been highly nontrivial due to the high-dimensionality and the complexity of the matrix.     
	
	(iii). It seems that the question of extending the least square regression in \cite{MR2906886} to the LAD with finite $2$-nd moment is still open. In this paper, we shall use a simple $1$-dimensional toy model to show that such a natural extension is impossible and the probability of the associated excess risk in the order $\mathcal{O}(\sqrt{1/n})$ is at most $1-\frac{9}{16 \rme^2}$. So a truncation is essential to mitigate the bad effect of outliers; see Section \ref{ss:NecTru}. When the dimension is high and the data are dependent, the heavy tail effect will make the LAD much worse. From our simulations for the LAD in the VAR($p$), we observe that the larger the $p$ is or the smaller the $\alpha$ is, the worse the performance of LAD will be.     
	
	(iv). We apply our result to the high-dimensional VAR($p$) models in which data only have $\alpha$-th moment with $\alpha \in (1,2]$. We simulate two VAR models, VAR($1$) and VAR($2$), and our simulation results show that our $\psi_\alpha$ truncation LAD remarkably outperforms both the classical LAD and the LAD with Huber truncation. The smaller the $\alpha$ is, the more significantly this improvement can be seen. We also apply our $\psi_\alpha$ truncation LAD to analyze the data set of the US GDP: (1) we first use Hill estimation to determine that the data is heavy-tailed with a finite $4$-th moment; (2) we fit the data into VAR($1$) model and compare the performance of our $\psi_\alpha$ truncation LAD, the classical LAD, and the LAD with Huber truncation. Our simulations indicate that the $\psi_\alpha$ truncation LAD performs the best among these three methods.
	
	The structure of this paper is outlined as follows. In the following section, we introduce our assumptions and main result. In \autoref{sec:Proof}, we establish our main result by utilizing auxiliary lemmas. In \autoref{sec:AR}, we apply our result to high-dimensional VAR models. To illustrate the advantages of \eqref{e:C-min}, we present simulations in \autoref{sec:simulation}, encompassing both VAR($p$) models and real time dependent data.
	
	\section{Preliminary and Main Results}
	\label{sec:sec2}
	
	\subsection{Notations, Assumptions and Problem Setting}
	
	Throughout this article, we use normal font for scalars (e.g.\ $a, A, \delta \dots$) and boldface for vectors and matrices (e.g.\ $\bfx,\bfy, \bfA,\bfB, \dots$). For any vector $\bfx = (x_1,\dotsc, x_d)^{\T}$ and $\bfy=(y_1, \dotsc, y_d)^{\T}$ in $\bbR^d$, let $\langle \bfx, \bfy \rangle = \sum_{i=1}^d x_i y_i $ be the inner product. Here, $\T$ denotes the transpose. This induces to the Euclidean norm $|\bfx| = \langle \bfx,\bfx \rangle^{1/2}$. 
	We denote the operator norm of any matrix $\bfA \in \bbR^{d_1 \times d_2}$ as $\opnorm{\bfA}= \sup_{ \bfx \in \mathbb{S}^{d_2-1} } \abs{ \bfA \bfx}$, where $\mathbb{S}^{d_2-1} = \{ \bfx\in\bbR^{d_2}: \abs{\bfx} = 1\}$, and denote  the rank of $\bfA$ by $\rank(\bfA)$, which satisfies $\rank(\bfA)  \le  d_1 \land d_2$. Let $\norm{\bfA}_{1,1} = \sum_{i = 1}^{d_1} \sum_{j=1}^{d_2} \abs{A_{ij}}$ be the $L_{1,1}$-norm of $\bfA$. If matrix $\bfB \in \bbR^{ d \times d}$ has eigenvalues $\lambda_i$, $i = 1\dotsc, d$, we let $\rho(\bfB) = \max_{1 \le  i  \le  d} \abs{ \lambda_i}$ be the spectral radius of $\bfB$. The following relations hold:
	\begin{equation}
		\label{e:op_11}
		\norm{\bfA}_{1,1}  \le  \sqrt{ ( d_1+d_2) \rank(\bfA) } \opnorm{\bfA} \quad
		\text{ and } \quad
		\rho(\bfB)  \le  \opnorm{\bfB}.
	\end{equation}
	For two sequences $\{f(n), n\in \mathbb{N}\} $ and $\{g(n), n\in \mathbb{N}\}$, $f(n) = \mathcal{O}(g(n))$ means $ f(n) / g(n) < \infty$ as $n \to \infty$ , and $f(n) = o(g(n))$ means $ f(n) / g(n) = 0$  as $n \to \infty$. 
	
	Suppose $\bfZ_1$ and $\bfZ_2$ are two real random variables (or vectors), and denote by $\bbP_{\bfZ_1 \times \bfZ_2}$ the joint probability of $(\bfZ_1, \bfZ_2)$ and by $\bbP_{\bfZ_1}, \bbP_{\bfZ_2}$ the probabilities of $\bfZ_1$ and $\bfZ_2$ respectively. Define their dependence as 
	\[
	\beta(\bfZ_1 , \bfZ_2) = \norm{ \bbP_{\bfZ_1 \times \bfZ_2} - \bbP_{\bfZ_1} \times \bbP_{\bfZ_2} }_{\mathrm{TV}},
	\]
	where $\norm{\cdot}_{\mathrm{TV}}$ is the total variation norm of probability measures. The coefficient $\beta(\bfZ_1 , \bfZ_2)$ vanishes if and only if $\bfZ_1$ and $\bfZ_2$ are independent. We refer the reader to \cite{lu2022almost,bradley2005basic,davydov1973mixing} for more details.

	Let $\{(\bfX_i, \bfY_i) \in \bbR^{d_1} \times \bbR^{d_2}, i \in \mathbb{N} \}$ under our consideration be a time series satisfying the following assumptions. 
	\begin{assumption}
		\label{A1:mixing}
		$\{(\bfX_i, \bfY_i), i  \ge  1\}$ is a stationary time series with a marginal distribution $\bm{\Pi}$, i.e., $(\bfX_i, \bfY_i) \sim \bm{\Pi}$ for each $i$, and exponentially $\beta$-mixing, i.e., there exist some constants $B > 0$ and $\beta > 0$ such that
		\[
		\beta(n)  \le  B \rme^{-\beta n} , \quad n \in \mathbb{N},
		\]
		where 
		\begin{equation*}
			\label{def:beta-mixing coefficient}
			\beta(n) = \sup_{k \ge  1} \beta \big(\{(\bfX_i, \bfY_i): i \le  k\}, \{(\bfX_i, \bfY_i): i \ge  n+k\} \big), \quad n \in \mathbb{N}.
		\end{equation*}		
	\end{assumption}
	We assume that $\bfX_i$ and $\bfY_i$ have linear relation up to external heavy-tailed noises. Since the time series is high-dimensional and stationary with marginal distribution $\bm{\Pi}$, it is natural to consider the following LAD problem: 
	\begin{equation}
		\label{e:PopLasso}
		\bmtheta^* = \argmin_{\bmtheta \in \bm{\Theta} }  [ R_{\ell_1}(\bmtheta) + \gamma \norm{\bmtheta}_{1,1} ]
		\quad 
		\text{with} \quad 
		\popR(\bmtheta) = \bbE_{(\bfX, \bfY) \sim \bm{\Pi}} \big[ \abs{\bfY- \bmtheta \cdot \bfX} \big],
	\end{equation}
	where $\bm{\Theta} \subseteq \bbR^{d_2 \times d_1}$ satisfies the assumptions below, the tuning parameter $\gamma$ in the penalty term will be chosen later, and $\bmtheta \cdot \bfX$ is a product of a matrix $\bmtheta \in \bbR^{d_2 \times d_1}$ with a vector $\bfX \in \bbR^{d_1}$. We emphasize that $|\cdot|$ is the Euclidean norm. The observed data for the minimization problem \eqref{e:PopLasso} are not i.i.d. but $\beta$-mixing time series, which is more reasonable in many real application scenarios. Let $\{(\bfX_i, \bfY_i): 1\le i \le n\}$ be the $n$ observations from the time series, it is natural to consider the corresponding empirical optimization problem:
	\begin{equation}
		\label{e:EmpMin}
		\bar{\bmtheta} = \argmin_{\bmtheta \in \bm{\Theta}} \big[ \widehat{R}_{\ell_1}(\bmtheta) + \gamma \norm{\bmtheta}_{1,1} \big] 
		\quad \text{with} \quad
		\widehat{R}_{\ell_1}(\bmtheta) = \frac{1}{n} \sum_{i = 1}^n \abs{ \bfY_{i} - \bmtheta \cdot \bfX_i } .
	\end{equation}
	
	We further assume that the following conditions regarding the $\alpha$-th moment hold.
	\begin{assumption}
		\label{A2:moments}
		\begin{itemize}
			\item[$(\runum{1})$] The $\alpha$-th moment of $\bfX$ with $\alpha \in (1,2]$ is bounded, that is,
			\[
			\bbE_{(\bfX, \bfY) \sim \bm{\Pi}}\big[ \abs{\bfX}^{\alpha} \big] < \infty.
			\]
			\item[$(\runum{2})$] The $\ell_{\alpha}$-risk with $\alpha \in (1,2]$  is uniformly bounded with respect to $\bmtheta \in \bm{\Theta}$, that is,
			\[
			\sup_{\bmtheta\in\bm{\Theta}} R_{\ell_{\alpha}} (\bmtheta)  < \infty, 
			\quad \text{with}\quad
			R_{\ell_{\alpha}} (\bmtheta) = \bbE_{(\bfX, \bfY) \sim \bm{\Pi}}\big[ \abs{ \bfY- \bmtheta \cdot \bfX }^{\alpha} \big] .
			\] 
		\end{itemize}
	\end{assumption}

	Besides these, we need assumptions on $\bm{\Theta}$ as below, and we refer the reader to \cite{MR3837109} for more details. 
	
	\begin{definition}
		Let $(\bm{\Theta}, \rmd)$ be a metric space, and $\bm{ \mathcal{K} }$ be a subset of $\bm{\Theta}$. Then, a subset $\bm{\mathcal{N} } \subseteq \bm{ \mathcal{K}}$ is called an $\delta$-net of $\bm{\mathcal{K}}$ if for any $\bmtheta \in \bm{\mathcal{K}}$, there exists a $\tilde{\bmtheta} \in \bm{\mathcal{N}}$ such that $\rmd(\bmtheta, \tilde{\bmtheta})  \le  \delta$. Besides, the covering number is the minimal cardinality of the $\delta$-net of $\bm{\Theta}$ and denoted by $N(\bm{\Theta},\delta)$.
	\end{definition}
	
	We shall assume that: 
	
	\begin{assumption}
		\label{A3:net}
		For any $\bmtheta, \tilde{\bmtheta} \in \bm{\Theta}$, let $\rmd(\bmtheta, \tilde{\bmtheta}) = \|{ \bmtheta - \tilde{\bmtheta} } \|_{\rm op}$. Assume that the domain $\bm{\Theta}$ is totally bounded, that is, for any $\delta>0$, there exists a finite $\delta$-net of $\bm{\Theta}$.
	\end{assumption}
	
	We also add the following sparsity assumption about the parameter $\bmtheta$, which will help us to compute the covering number $N(\bm{\Theta},\delta)$. 	
	\begin{assumption}
		\label{A4:matrix}
		For any $\bmtheta \in \bm{\Theta} \subseteq \bbR^{d_2 \times d_1}$, we assume that $\rank(\bmtheta)  \le  \kappa$ and $\opnorm{ \bmtheta } \le  R$ for some constants $\kappa, R > 0$.
	\end{assumption}
	
	\subsection{Necessity of Truncation} \label{ss:NecTru}
	Recall that \cite{MR2906886} studied a ridge least square regression for high-dimensional data with finite $4$-th moment, and proved that their model has an excess risk in the order $d/n$. It is natural to extend this result to the LAD with finite $2$nd moment and obtain an excess risk in the order
	$\mathcal{O}(\sqrt{d/n})$. The following toy model will tell us that such an extension is impossible, so a truncation is necessary. 
	
	We consider the following model on $\bbR$:
	\begin{equation} \label{e:Model}
		Y = \theta^*+\zeta, 
	\end{equation} 
	where we assume that the true value $\theta^*=0$ without loss of generality. Let $n$ be the number of the observed data from the model. We assume $  \zeta $ follows a distribution with a density probability function given by  
	\begin{equation}\label{e:s_noise}
		p_{  \zeta }(x) = \rho_n \mathbbm{1}_{\{x=0\}} +  \frac{(1-\rho_n) \mu}{2\abs{x}^{\mu+1}} \mathbbm{1}_{ \{ \abs{x} \ge  1 \} }
	\end{equation}
	for some constant $\mu \in (1,2]$ and $\rho_n = 1/\sqrt{n}$. It is easy to verify that $\bbE [ Y ] = 0$, $\bbE [ \abs{Y}^\alpha ] < \infty$ for all $\alpha < \mu$.
	It is clear that 
	\begin{equation}\label{e:s_model}
		\quad \theta^{*} = \argmin_{\theta} R_{\ell_1}(\theta ) \quad \text{with} \quad R_{\ell_1}(\theta ) =  \bbE \abs{Y-\theta}.
	\end{equation} 
	Furthermore, a straightforward calculation yields that
	\begin{equation}\label{e:R_1}
		R_{\ell_1}(\theta) 
		=\left\{
		\begin{aligned}
			\frac{\mu}{\mu-1} (1-\rho_n) + \rho_n \abs{\theta}, \qquad \qquad & \abs{\theta} < 1 , \\
			(1-\rho_n)\left[ \abs{\theta} + \frac{1}{\mu-1} \abs{\theta}^{-\mu+1} \right] + \rho_n \abs{\theta}, \quad &\abs{\theta} \ge  1 ,
		\end{aligned}
		\right.
	\end{equation} 
	which is an even function with respect to $\theta$ and is increasing for $\theta>0$. Thus, 
	\[
	R_{\ell_1}(\theta^*) = R_{\ell_1}(0) = \frac{\mu}{\mu-1}(1-\rho_n).
	\] 
	
	Let $\{ Y_i, i =1,\dotsc, n \} $ be a sequence of observations from the model \eqref{e:Model}, which are i.i.d. and have the distribution \eqref{e:s_noise}. 
	Recall that the LAD estimator for $\theta^*$ is defined as 
	\[
	\bar{\theta} = \argmin_{\theta} \widehat{R}_{\ell_1}(\theta), \quad \widehat{R}_{\ell_1}(\theta) = \frac{1}{n} \sum_{i = 1}^{n} \abs{Y_i - \theta}.
	\]
	In order to make our argument a little more simple, we assume $n=2m+1$. The other case $n=2m$ can be analysed similarly. 
	It is well known that 
	$$\bar{\theta} = Y_{(m+1)}$$ is the median of  $\{ Y_i, i =1,\dotsc, 2m+1 \}$. We shall rigorously show in the \autoref{s:LADBigPro} that as $n \ge n_0$ with some constant $n_0 \in \mathbb{N}$, we have
	\begin{equation} \label{e:LADBigPro}
		\bbP\left( R_{\ell_1}(\bar{\theta}) - R_{\ell_1}(\theta^*)  \ge   \frac{1}{\sqrt{n}} \right)  \ge  \frac{9}{16 \rme^2}.
	\end{equation} 
	
	From \eqref{e:LADBigPro}, we know that in order to make the LAD estimation work well, one has to sample a large number of observations to reduce the bad effects of outliers. A suitable truncation is needed to remove outliers and make the estimation more efficient.   
	
	\subsection{Our Truncated Estimation and Main Result} Our Catoni type truncation function $\psi_\alpha(r): \bbR \to \bbR$ satisfies that
	\begin{equation}
		\label{e:T-func}
		-\log\left( 1 - r + \frac{\abs{r}^{\alpha}}{\alpha} \right)  \le  \psi_\alpha(r)  \le  \log\left( 1 + r + \frac{\abs{r}^{\alpha}}{\alpha} \right), \quad \alpha \in (1,2].
	\end{equation}
	Let  $\{(\bfX_i, \bfY_i), 1 \le  i  \le  n\}$ be the observed time series, we propose the following truncated estimation: 
	\begin{equation}
		\label{e:C-min}
		\begin{split}
			\hat{\bmtheta} &= \argmin_{\bmtheta \in \bm{\Theta} } \big[ \widehat R_{\psi_\alpha, \ell_1} (\bmtheta) + \gamma \norm{\bmtheta}_{1, 1} \big] \\
			\widehat R_{\psi_\alpha, \ell_1}(\bmtheta) &= \frac{1}{n\lambda} \sum_{i=1}^n \psi_\alpha \big( \lambda \abs{\bfY_i - \bmtheta \cdot \bfX_i } \big),
		\end{split}
	\end{equation}
	where $\lambda, \gamma>0$ are the tuning parameters to be chosen later.
	\begin{definition}
		\label{def:excess risk}
		Generally speaking, excess risk is the gap of population risk between the current model and the optimal one.	
		In this paper, we consider the popular risk $R_{\ell_1}$ given in \eqref{e:PopLasso}, and the corresponding excess risk is $$\popR(\hat{\bmtheta}) - \popR(\bmtheta^*),$$ where $\bmtheta^*$ and  $\hat{\bmtheta}$ are defined in \eqref{e:PopLasso} and \eqref{e:C-min} respectively. 
	\end{definition}
	The following theorem is our main theoretical result in this paper, whose proof will be given in the following section.  
	
	\begin{theorem}
		\label{cor:M_n = m_n = log n}
		Let $\bmtheta^*$ and $\hat{\bmtheta}$ be the minimizers of minimization problems \eqref{e:PopLasso} and \eqref{e:C-min} respectively. Under Assumptions \ref{A1:mixing}, \ref{A2:moments}, \ref{A3:net} and \ref{A4:matrix}, for any $\varepsilon \in (0,1/2)$, let the parameters $\delta$, $\lambda$ and $\gamma$   satisfy 
		\[
		\delta = \frac{12 \log n}{n\beta}, \quad 
		\lambda = \left( 2\delta \left( \log \frac{16}{\varepsilon^2} + (d_1 + d_2) \kappa \log \frac{6 R}{\delta} \right) \right)^{{1}/{\alpha}}, \quad
		\gamma = \frac{\log n}{n }, 
		\]
		where $n$ is the sample size sufficiently large such that
		\[
		\delta < 1, \quad	\frac{(d_1 + d_2 )\kappa \log n}{n} < 1 \quad \text{ and } \quad \frac{\beta}{n(2\log n - \beta)}  \le  \frac{\varepsilon}{2 B } \left( \frac{\delta}{6R} \right)^{ (d_1+d_2) \kappa } .
		\]
		Then, the following inequality holds with probability at least $1-2\varepsilon$,
		\begin{equation}
			\label{ineq:2 of cor.2}
			\popR(\hat{\bmtheta}) - \popR(\bmtheta^*)  \le  C \left( \frac{ \log n  }{\beta n} \big( \abs{\log \varepsilon} + (d_1+d_2) \kappa \log n \big) \right)^{{(\alpha-1)} / {\alpha}},
		\end{equation} 
		for some constant $C>0$ independent of $n, \varepsilon, d_1, d_2, \kappa$ and $\alpha$. 
	\end{theorem}
	\begin{remark}
		The absolute loss $\ell_1(\bfy, \bfx, \bmtheta) = \abs{\bfy- \bmtheta \cdot \bfx}$ also can be replaced by all Lipschitz losses \cite{MR3953446} $\ell(\bfy, \bfx, \bmtheta)$ satisfying
		\[
		\abs{ \ell( \bfy, \bfx, \bmtheta_1 ) - \ell( \bfy, \bfx, \bmtheta_2 ) }  \le  L \abs{  (\bmtheta_1 - \bmtheta_2 ) \cdot \bfx }, \quad \forall \ \bmtheta_1, \bmtheta_2 \in \bm{\Theta}. 
		\]
		for some constant $L$. For instance, when $d_1=d$ and $d_2=1$, the $\tau$-quantile loss is defined as $\ell( y, \bfx, \bmtheta) = \rho_{\tau}(y - \inprod{\bfx, \bmtheta})$, where $\rho_{\tau}(u) = u [ \tau - \mathbbm{1}(u < 0) ] $ for all $u\in \bbR$ with $\tau \in (0,1)$. This corresponds to quantile regression.
	\end{remark}

	\section{Proof of Theorem \ref{cor:M_n = m_n = log n}}
	\label{sec:Proof}
	
	\subsection{The Strategy of The Proof.}
	
	Let us introduce the strategy of proving Theorem \ref{cor:M_n = m_n = log n} and provide some auxiliary lemmas in this section. The proof can be decomposed into the following three key ingredients.  
	
	(i) \underline{Excess risk decomposition}: We write 
	\begin{equation} \label{e:decom}
		\begin{split}
			&\pheq
			\popR(\hat{\bmtheta}) - \popR(\bmtheta^*) \\
			&= 
			\popR(\hat{\bmtheta})  - \catoniR(\hat{\bmtheta}) 
			+ \catoniR(\hat{\bmtheta}) - \catoniR(\bmtheta^*) 
			+ \catoniR(\bmtheta^*) - \popR(\bmtheta^*) \\
			& \le  \left[ \popR(\hat{\bmtheta})  - \catoniR(\hat{\bmtheta}) \right]
			+\left[  \catoniR(\bmtheta^*) - \popR(\bmtheta^*) \right]
			+ \gamma \norm{\bmtheta^*}_{1, 1},
		\end{split}
	\end{equation}
	where the inequality is derived as the following: for $\bmtheta^* \in \bm{\Theta}$, the definition of $\hat{\bmtheta}$ in \eqref{e:C-min} yields that
	\[
	\catoniR(\hat{\bmtheta}) + \gamma \Vert \hat{\bmtheta} \Vert_{1,1}  \le  \catoniR(\bmtheta^*) + \gamma \norm{\bmtheta^*}_{1,1},
	\]
	which leads to $\catoniR(\hat{\bmtheta}) - \catoniR(\bmtheta^*) \le  \gamma \norm{\bmtheta^*}_{1,1}$.
	
	(ii) \underline{Bounding $\catoniR(\bmtheta^*) - \popR(\bmtheta^*)$}: We shall use the block technique to handle the problems arising from the dependence of the time series data, see more details in Lemma \ref{lemma:error bound w.r.t. theta star} below. 
	
	(iii) \underline{Bounding $\popR(\hat{\bmtheta})  - \catoniR(\hat{\bmtheta})$}: We shall use the $\delta$-net and finite covering technique, together with the aforementioned block technique, to bound $\popR(\hat{\bmtheta})  - \catoniR(\hat{\bmtheta})$, see more details in Lemma \ref{lemma: error bound of theta hat}.
		
	\subsection{The Block Technique} We introduce the block technique in detail. Divide $\{1,\dotsc, n\}$ into blocks as defined in \eqref{def:blocks} and \eqref{def:reserved block} with length $M_n$, $m_n$ and $R_n$ respectively. Let $M_n$ and $m_n$ denote the length of the big and small blocks respectively, with $m_n  \le  M_n$. Define 
	\[
	K(n) = \sup_{k \ge 1} \{ k \in \mathbb{Z}: k(M_n + m_n)  \le  n  \}.
	\]
	Then for any $1 \le  j  \le  K(n)$, put
	\begin{equation}
		\label{def:blocks}
		\begin{aligned}
			\mathcal{J}_{j,M_n} &= \left\{ i : (j-1) (M_n+m_n) + 1  \le  i  \le  (j-1) (M_n+m_n) + M_n \right\} , \\
			\mathcal{I}_{j,m_n} &= \left\{ i : (j-1) (M_n+m_n) + M_n + 1  \le  i  \le  j (M_n + m_n) \right\} ,
		\end{aligned}
	\end{equation}
	where $J_{j,M_n}$ (resp., $I_{j,m_n} $) are big (resp., small) blocks. The tail block is given by
	\begin{equation}
		\label{def:reserved block}
		\mathcal{R}_{R_n} = \left\{ i : K(n) (m_n + M_n) + 1 \le  i  \le  n \right\},
	\end{equation}
	which leads to $\abs{\mathcal{R}_{R_n}} = R_n \in [0,M_n + m_n)$. 	
	We further denote that
	\[
	\mathcal{J}_{M_n} = \bigcup_{j} \mathcal{J}_{j,M_n}\ , \quad \mathcal{I}_{m_n} = \bigcup_{j} \mathcal{I}_{j,m_n} ,
	\]
	which are the collection of big blocks and small blocks respectively, such that $\{ 1, \dotsc, n \} = \mathcal{J}_{M_n} \cup \mathcal{I}_{m_n} \cup \mathcal{R}_{R_n}$. Let $\{ \bfZ_i = (\bfX_i , \bfY_i) , i \in \mathbb{N} \}$ be a $\bbR^{d_1+d_2}$-valued time series, and let 
	\begin{equation} 
		\label{def:HGR}
		\bm{\mathcal{H}}_j = \big( \bfZ_i, i \in J_{j,M_n} \big) , \quad 
		\bm{\mathcal{G}}_j = \big( \bfZ_i , i \in I_{j,m_n} \big) 
		\quad \text{and} \quad
		\bm{ \mathcal{R}} = \big( \bfZ_i, i \in  \mathcal{R}_{R_n}\big) .
	\end{equation}	
	
	By \cite[Lemma 2.1]{MR0871254}, we can construct independent random vector sequence $\widetilde{\bm{\mathcal{H}}}_{1}, \dotsc, \widetilde{\bm{\mathcal{H}}}_{K(n)}$, $ \widetilde{\bm{\mathcal{R}}} $, which satisfy that $\widetilde{\bm{\mathcal{H}}}_{i}$ and ${\bm{\mathcal{H}}}_{i}$ have the same distribution for each $1 \le i \le K(n)$ and that $ \widetilde{\bm{\mathcal{R}}} $ and $ \bm{ \mathcal{R}} $ have the same distribution. For the remaining small blocks, by the same theorem, we can construct independent random vectors $\widetilde{\bm{\mathcal{G}}}_{1}$,...,$\widetilde{\bm{\mathcal{G}}}_{K(n)}$ which satisfy that $\widetilde{\bm{\mathcal{G}}}_{i}$ and ${\bm{\mathcal{G}}}_{i}$ have the same distribution for each $1 \le i \le K(n)$. Note that these two random vector sequences are not necessarily independent. Conveniently, we define them as
	\begin{equation}
		\label{def:ind HGR}
		\widetilde{\bm{\mathcal{H}}}_j = \big( \widetilde{\bfZ}_i, i \in J_{j,M_n} \big) , \quad 
		\widetilde{\bm{\mathcal{G}}}_j = \big( \widetilde{\bfZ}_i , i \in I_{j,m_n} \big) , 
		\quad \text{and} \quad
		\widetilde{\bm{ \mathcal{R}} } = \big( \widetilde{\bfZ}_i, i \in  \mathcal{R}_{R_n}\big) .
	\end{equation}	 
	Define events  $\mathcal{A}$ and $\mathcal{B}$ as
	\begin{equation} \label{def: events A & B}
		\begin{aligned}
			\mathcal{A} &= \big\{ \widetilde{\bm{\mathcal{H}}}_{j} \neq \bm{\mathcal{H}}_{j} \text{ for some } 1 \le  j  \le  K_n , \ \text{or} \ \widetilde{\bm{\mathcal{R}}} \neq \bm{ \mathcal{R}} \big\} , \\
			\mathcal{B} &= \big\{ \widetilde{\bm{\mathcal{G}}}_{j} \neq \bm{\mathcal{G}}_{j} \text{ for some } 1 \le  j  \le  K_n \big\} .
		\end{aligned}
	\end{equation}
	Then, under Assumption \ref{A1:mixing}, \cite[Lemma 2.1]{MR0871254} yields that
	\[
	\bbP\big( \mathcal{A} \big)  \le  K(n) \beta(m_n)  \le  BK(n)\rme^{-\beta m_n}  
	\quad \text{and} \quad
	\bbP\big( \mathcal{B} \big)  \le  K(n) \beta(M_n)  \le  BK(n)\rme^{-\beta M_n}.
	\]
	
	\subsection{Auxiliary Lemmas}
	The usage of our truncation function $\psi_\alpha(r)$ satisfying \eqref{e:T-func} is to obtain the following lemma: 
	
	\begin{lemma}
		\label{lem:C_Ine}
		For any subset $ \mathcal{I} \subseteq \{1,\dotsc, n\}$, and any constants $\lambda, \delta>0$, we have the following inequalities:
		\begin{equation}
			\label{e:upper}
			\bbE\left[ \exp\bigg\{ \frac{1}{\abs{\mathcal{I}}} \sum_{i \in \mathcal{I}}  \psi_{\alpha}  \big( \lambda \abs{ \bfY_i - \bmtheta \cdot \bfX_i  } \big) \bigg\} \right]
			\le  
			\exp\Big\{ \lambda R_{\ell_1}(\bmtheta) + \alpha \lambda^{\alpha} R_{\ell_{\alpha}}(\bmtheta) \Big\} ,
		\end{equation}
		and 
		\begin{equation}
			\label{e:lower}
			\begin{aligned}
				&\pheq
				\bbE\left[ \exp\bigg\{ -\frac{1}{\abs{\mathcal{I}}} \sum_{i \in \mathcal{I}}  \psi_{\alpha} \left( \lambda \abs{\bfY_i - \bmtheta \cdot \bfX_i }  - \lambda \delta \abs{\bfX_i} \right) \bigg\} \right] \\
				& \le  
				\exp\left\{ \lambda \Big[ - \popR(\bmtheta) + \delta \bbE\abs{\bfX_1} + \frac{(2\lambda)^{\alpha-1}}{\alpha} \Big( \sup_{\bmtheta\in\bm{\Theta}} R_{\ell_{\alpha}}(\bmtheta) + \delta^{\alpha} \bbE\abs{\bfX_1}^{\alpha} \Big) \Big] \right\} .
			\end{aligned}
		\end{equation}
	\end{lemma}
	
	Combining Lemma \ref{lem:C_Ine} with the block technique, we can obtain the following lemmas:
	
	\begin{lemma}
		\label{lemma:error bound w.r.t. theta star}
		Under Assumptions \ref{A1:mixing} and \ref{A2:moments}, define $M_n, m_n$ and $K(n)$ as defined in \eqref{def:blocks}. For any $\varepsilon \in (0,1)$, let $n$ be large enough such that
		\begin{equation}\label{e:cond.1}
			K(n) \rme^{-\beta m_n}  \le  \frac{\varepsilon}{4B}
			\quad \text{and } \quad M_n = o\big(K(n)m_n\big). 
		\end{equation} 
		Then the following inequality holds with probability at least $1-\varepsilon$,
		\[
		\catoniR(\bmtheta^*) - \popR(\bmtheta^*)  \le  \frac{\lambda^{\alpha-1}}{\alpha} R_{\ell_{\alpha}} (\bmtheta^*) + \frac{2 M_n}{\lambda K(n)m_n} \log\frac{4}{\varepsilon} .
		\]	
	\end{lemma}
	
	\begin{lemma}
		\label{lemma: error bound of theta hat}
		Under Assumptions \ref{A1:mixing} and \ref{A2:moments}, define $M_n, m_n$ and $K(n)$ as defined in \eqref{def:blocks}. Let Assumption \ref{A3:net} hold additionally and let $\bm{\mathcal{N}}(\bm{\Theta},\delta)$ be an $\delta$-net of $\bm{\Theta}$ with cardinality $N(\bm{\Theta},\delta)$ for any $\delta>0$. For any $\varepsilon\in(0,1)$, let $n$ satisfy that
		\begin{equation}\label{e:cond.2}
			K(n) \rme^{-\beta m_n}  \le  \frac{\varepsilon}{4B N(\bm{\Theta},\delta)}
			\quad \text{and } \quad M_n = o\big(K(n)m_n\big) .
		\end{equation}
		Then, the following inequality holds with probability at least $1-\varepsilon$,
		\[
		\begin{aligned}
			\popR(\hat{\bmtheta}) - \catoniR(\hat{\bmtheta}) 
			\le 
			\delta \bbE\abs{\bfX_1} + h(n,\alpha,\lambda, \delta, \varepsilon).
		\end{aligned}
		\]
		Here, function $h$ is defined as
		\begin{equation}
			\label{eq:h}
			\begin{aligned}
				h(n,\alpha,\lambda,\delta,\varepsilon)
				&= \delta \bbE\abs{\bfX_1} + \frac{(2\lambda)^{\alpha-1}}{\alpha}\left( \sup_{\bmtheta\in\bm{\Theta}} R_{\ell_{\alpha}}(\bmtheta) + \delta^{\alpha} \bbE\abs{\bfX_1}^{\alpha} \right) \\
				&\pheq +  \frac{2 M_n}{\lambda K(n)m_n} \log\frac{4 N(\bm{\Theta},\delta)}{\varepsilon} .
			\end{aligned}
		\end{equation}
	\end{lemma}
	
	All above lemmas are proved in \autoref{appendix}.
	
	\subsection{Proof of Theorem \ref{cor:M_n = m_n = log n}}
	
	Now we can prove our main result by applying Lemmas \ref{lemma:error bound w.r.t. theta star} and \ref{lemma: error bound of theta hat} with special $M_n$ and $m_n$.
		
	\begin{proof}[\textbf{Proof of Theorem \ref{cor:M_n = m_n = log n}}]
		We select special $M_n$ and $m_n $ as following: 
		\[
		M_n = m_n = \floor{\frac{2 }{\beta} \log n} \in \left[ \frac{2}{\beta}\log n-1\ , \ \frac2\beta\log n \right].
		\]
		Then, we have
		\[
		\frac{\beta n-4\log n}{ 4 \log n}  \le  K(n)  \le  \frac{\beta n}{2(2\log n-\beta)},
		\]
		which leads to
		\begin{gather*}
			K(n) \rme^{-\beta m_n}  \le  \frac{\beta n\ \rme^{-2 \log n} }{2(2\log n-\beta)} = \frac{\beta}{2n(2\log n -\beta)} , \\
			\frac{2\log n- \beta}{\beta n}  \le  \frac{M_n}{K(n) m_n}   \le  \frac{12 \log n}{\beta n}.
		\end{gather*}
		Hence, for any $\varepsilon>0$ and $\delta >0$, there exists a constants $n_0 $ such that for all $n \ge  n_0$
		\[
		K(n) \rme^{-\beta m_n}  \le  \frac{\beta}{2n(2\log n -\beta)}  \le  \frac{\varepsilon}{4BN((\bm{\Theta},\delta))} ,
		\]
		and $M_n = o(K(n)m_n)$. That is, the conditions \eqref{e:cond.1} and \eqref{e:cond.2}  hold, and we can apply Lemmas \ref{lemma:error bound w.r.t. theta star} and \ref{lemma: error bound of theta hat} to obtain the theorem. 
		
		It follows from Lemma \ref{lemma:error bound w.r.t. theta star} that the following inequality holds with probability at least $1-\varepsilon$
		\begin{equation}
			\label{e:bound*}
			\catoniR(\bmtheta^*) - \popR(\bmtheta^*)  \le  \frac{\lambda^{\alpha-1}}{\alpha} R_{\ell_{\alpha}} (\bmtheta^*) + \frac{24 \log n}{\lambda \beta n} \log\frac{4}{\varepsilon}.
		\end{equation}
		By Lemma \ref{lemma: error bound of theta hat}, 
		\begin{equation}
			\begin{aligned}
				\label{e:bound_hat}
				\popR(\hat{\bmtheta}) - \catoniR(\hat{\bmtheta}) 
				& \le 
				2 \delta \bbE\abs{\bfX_1} + \frac{(2\lambda)^{\alpha-1}}{\alpha}\left( \sup_{\bmtheta\in\bm{\Theta}} R_{\ell_{\alpha}}(\bmtheta) + \delta^{\alpha} \bbE\abs{\bfX_1}^{\alpha} \right) \\
				&\pheq +  \frac{24 \log n}{\lambda \beta n} \log\frac{4 N(\bm{\Theta},\delta)}{\varepsilon}
			\end{aligned}
		\end{equation}
		holds with probability at least $1-\varepsilon$. Combining \eqref{e:decom}, \eqref{e:bound*} and \eqref{e:bound_hat} immediately implies that
		\begin{equation}
			\label{ineq:1 of cor.2}
			\begin{aligned}
				\popR(\hat{\bmtheta}) - \popR(\bmtheta^*) 
				& \le  \left[ {2^{\alpha-1} \delta^{\alpha}}\bbE\abs{\bfX_1}^{\alpha} + [ { 1 + 2^{\alpha-1}} ] \sup_{\bmtheta\in\bm{\Theta}} R_{\ell_{\alpha}}(\bmtheta) \right] \frac{\lambda^{\alpha-1}}{\alpha} \\
				&\pheq + \frac{ 24 \log n}{\lambda \beta n} \log \frac{16N(\bm{\Theta},\delta)}{\varepsilon^2} +  2\delta\bbE\abs{\bfX_1}  + \gamma \norm{\bmtheta^*}_{1,1}.
			\end{aligned}
		\end{equation}
		holds with probability at least $1-2\varepsilon$. Furthermore, by equation \eqref{e:op_11} and \cite[Lemma 7]{ZhangAnru8368145}, Assumption \ref{A4:matrix} yields that for any $\delta \in (0, 1]$
		\[
		N(\bm{\Theta},\delta)  \le  \left( \frac{6R}{\delta} \right)^{(d_1 + d_2)\kappa}, \quad 
		\norm{\bmtheta^*}_{1,1}  \le  \sqrt{(d_1+d_2) rank(\bmtheta^*)} \|\bmtheta^*\|_{\rm op}  \le  \sqrt{(d_1+d_2) \kappa} R.
		\]	
		We select the tuning parameters $\delta $, $\lambda$ and $\gamma$ in \eqref{ineq:1 of cor.2} as
		\[
		\delta = \frac{12\log n}{\beta n} , \quad 
		\lambda = \left( 2 \delta \left( \log \frac{16}{\varepsilon^2} + (d_1+d_2)\kappa\log \frac{5 R}{\delta} \right) \right)^{{1}/{\alpha}}
		\quad  \text{and} \quad 
		\gamma = \frac{\log n}{n}.
		\]
		Let $n$ be sufficiently large such that $\delta, \gamma \in (0,1)$. It follows from \eqref{ineq:1 of cor.2} that
		\[
		\begin{aligned}
			&\pheq
			\popR(\hat{\bmtheta}) - \popR(\bmtheta^*) \\
			& \le  \left[   \frac{ 3 }{\alpha} \left( \sup_{\bmtheta\in\bm{\Theta}} R_{\ell_{\alpha}}(\bmtheta) + \bbE\abs{\bfX_1}^{\alpha} \right)+ 1 \right]\!\! \left[ \frac{ 24 \log n}{\beta n} \left( \log \frac{16}{\varepsilon^2} + (d_1+d_2)\kappa\log \frac{5R}{\delta} \right)  \right]^{(\alpha-1)/\alpha} \\
			&\pheq + \frac{24\log n}{\beta n} \bbE \abs{\bfX_1} + \frac{ \sqrt{(d_1+d_2) \kappa} \log n }{n} R\\
			& \le  C \left( \frac{\log n}{\beta n} \big( \abs{\log \varepsilon} + (d_1 + d_2) \kappa \log n \big)\right)^{(\alpha - 1)/\alpha}
		\end{aligned}
		\]
		holds with probability at least $1-2\varepsilon$,  for some constant $C>0$ independent of parameters $\varepsilon, n, d_1, d_2, \kappa$ and $\alpha$. 
	\end{proof}
	
	\begin{remark}
		For any fixed and sufficiently large $n$, different choices of $M_n$ and $m_n$ in the proof of Theorem \ref{cor:M_n = m_n = log n} will lead to different convergence rates. For example, if we choose $M_n = m_n = \mathcal{O}(n^{\sigma})$ for some constant $\sigma \in (0,1)$, the same argument can yield that
		\begin{equation}\label{e:remark}
			\popR(\hat{\bmtheta}) - \popR(\bmtheta^*)  \le  C \left( \frac{(1-\sigma) }{n^{1-\sigma}} \big( \abs{\log \varepsilon} + (d_1+d_2)\kappa  \log n \big) \right)^{{(\alpha-1)} / {\alpha}}
		\end{equation}
		holds with probability at least $1-2\varepsilon$ for some constant $C>0$. When the data are independent, the $\beta$-mixing coefficient will vanish. Thus, we can take $M_n = m_n = 1$, that is, $\sigma=0$ in \eqref{e:remark}, and get  
		\begin{equation*}
			\popR(\hat{\bmtheta}) - \popR(\bmtheta^*)  \le  C \left( \frac{1}{n} \big( \abs{\log \varepsilon} + (d_1+d_2) \kappa  \log n \big) \right)^{{(\alpha-1)} / {\alpha}}
		\end{equation*}
		holds with probability at least $1-2\varepsilon$.
	\end{remark}
	
	\section{Application to High-Dimensional Vector Autoregression}
	\label{sec:AR}
	
	Autoregressive modeling of multivariate stationary processes were originally introduced in control theory, where vector-valued autoregressive moving average and state-space representation were used as canonical tools for identification of linear dynamic system (see, \cite{Kumar1986,Hannan2012} ). Additionally, a large class of stationary processes can be represented as potentially infinite order vector autoregression (VAR); see for instance \cite{Fournier2006, MR2409258, MR2172368, 8600392, MR4480716, networkVA, MR4726961}.
	
	Formally, for a $d$-dimensional time series $\bfZ_t \in \bbR^d$, a finite-order VAR model of order $p$, often denoted as VAR$(p)$, is given by the form
	\begin{equation}
		\label{e:VAR(p)}
		\bfZ_{t+1} = \bmPhi_1 \cdot \bfZ_t + \bmPhi_2 \cdot \bfZ_{t-1} + \cdots + \bmPhi_{p} \cdot \bfZ_{t+1-p} + \bm{\varepsilon}_{t+1}, \quad t \in \mathbb{N},
	\end{equation}
	where $\bmPhi_i \in \bbR^{d\times d}$, $i=1,\dotsc,p$ and $\bmepsilon_{t+1}$ is a noise process. In this paper, we add the following assumptions for \eqref{e:VAR(p)}:
	
	\begin{assumption}
		\label{A5:noise}
		The noises $\{ \bmepsilon_t \}_{t \ge 1}$ are i.i.d. and admit a density $g$. We assume that for all $t$:
		\begin{itemize}
			\item[$(\runum{1})$]  There exists a constant $M>0$ such that $\int \abs{  g(\bfx - \bfy) - g(\bfx) \dd \bfx} < M \abs{ \bfy}$ for all $\bfy \in \bbR^d$.
			\item[$(\runum{2})$] Additionally, $\bmepsilon_t$ is centered and has bounded $\alpha$-th moment with $\alpha \in (1,2]$, i.e., 
			\[
			\bbE[\bmepsilon_{t}] = 0, \quad \text{and} ,\quad \bbE \abs{\bmepsilon_{t}}^\alpha  \le  K < \infty,  
			\]
			for some positive constant $K$.
		\end{itemize} 
	\end{assumption}
	
	\begin{assumption}
		\label{A6:stationary} 
		Define a matrix $\bmPsi \in \bbR^{(dp) \times (dp)}$ as
		\[
		\bmPsi = 
		\begin{bmatrix}
			\bmPhi_1   & \cdots  & \bmPhi_{p-1} & \bmPhi_{p} \\
			\bfI_d       & \cdots  &    \zero     & \zero \\
			\vdots     & \ddots  &    \vdots    & \vdots \\
			\zero      & \cdots  &    \bfI_d      & \zero
		\end{bmatrix}
		\]
		with $\bmPhi_1  \cdots   \bmPhi_{p}$ defined in \eqref{e:VAR(p)}, and $\bfI_d$ being the identity matrix in $\bbR^{d\times d}$. We assume that $\rho(\bmPsi)  \le  \rho$  for some constant $\rho \in (0,1)$.
	\end{assumption} 
	
	Under Assumption \ref{A6:stationary}, for all $n\in \mathbb{N}$, there is a constant $C_{\rm op} >0$ independent of $n$ such that $\opnorm{ \bmPsi^n}  \le  C_{\rm op} \rho^n$. Additionally, it holds that
	\begin{equation} \label{e:theta_VAR}
		\opnorm{ [\bmPhi_1, \cdots, \bmPhi_{p}] }  \le  \opnorm{\bmPsi}  \le  C_{\rm op} \rho.
	\end{equation} 
	The following lemma arises from Theorem \ref{thm:mixing condition} in \autoref{appendix: mixing} , and yields that time series $\{\bfZ_t\}_{t \in \mathbb{N}}$ is exponentially $\beta$-mixing. 
	
	\begin{lemma}\label{lem:mixing VARp}
		Let Assumption \ref{A6:stationary} hold, and assume that the eigenvalues of $\bmPsi$ are not all zero additionally. Let the noise ${\bmepsilon_t}$ satisfy Assumption \ref{A5:noise}. Then, there exists a constant $B>0$ such that the $\beta$-mixing coefficient of $\{\bfZ_t\}_{t \in \mathbb{N}}$ satisfies
		\begin{equation*}
			\beta(n)  \le  B \cdot \exp\left\{ -\frac{1}{2} \abs{\log {\rho}} n\right\}. 
		\end{equation*}
	\end{lemma}
	
	More details can be found in \autoref{appendix: mixing}. Furthermore, let $\{x_i \in \mathbb{C}, i = 1, \dotsc, d\}$ be the roots of equation
	\begin{equation*}
		\begin{aligned}
			\det(  \bfI - x \bmPsi)
			&= \det\left( { \mathbf{I} - x \bmPhi_1 - \cdots - x^p \bmPhi_p } \right) =0 .
		\end{aligned}
	\end{equation*} 
	Due to $\rho(\bmPsi)   \le  \rho < 1$ under Assumption \ref{A6:stationary},  it holds that $\abs{x_i}  \ge  1/\rho$ for all $i=1, \dotsc ,d$. That is, all roots $x_i$ lie outside the unit circle. This immediately implies that \eqref{e:VAR(p)} will be stationary according to Wei \cite{Wei2019}. That is, there exists an unknown distribution $\bm{\Pi}$ such that $(\bfZ_t, \dotsc, \bfZ_{t+p}) \sim \bm{\Pi}$ for all $t$. 
	
	Let us now fit this VAR(p) model into the framework of LAD. Let us $\bfY = \bfZ_{p+1}$ and $\bfX = ( \bfZ_p, \dotsc, \bfZ_{1} )$, it follows from Assumption \ref{A6:stationary} that $(\bfX, \bfY) \sim \bm{\Pi}$ for an unknown distribution $\bm{\Pi}$. Then, the classical LAD is defined as
	\begin{equation}
		\label{Pop-VAR}
		\begin{aligned}
			[\bmPhi_1^*, \dotsc, \bmPhi_{p}^*] &= \argmin_{[\bmPhi_1, \cdots, \bmPhi_{p}] \in \bm{\Theta}} \left\{  R_{\ell_1}(\bmPhi_1, \cdots, \bmPhi_{p}) + \gamma \sum_{i = 1}^p \norm{ \bmPhi_i }_{1,1} \right\}, \\
			R_{\ell_1}(\bmPhi_1,\dotsc,\bmPhi_{p}) &= \bbE_{(\bfX, \bfY) \sim \bm{\Pi}} \left[ \abs{\bfY- [\bmPhi_1, \cdots, \bmPhi_{p}] \cdot \bfX } \right] .
		\end{aligned}
	\end{equation}
	For VAR$(p)$ model, we regard that $\bm{\Theta}$ is the collection of matrices given by 
	\begin{equation} \label{e:Xi}
		\bm{\Theta} = \{ \bmtheta = [\bmPhi_1, \cdots, \bmPhi_{p}] \in \bbR^{d\times pd}:  
		\opnorm{\bmtheta}  \le  C_{\rm op} \rho \},
	\end{equation}	
	for some constants $\rho$ in Assumption \ref{A6:stationary} and $C_{\rm op}$ in equation \eqref{e:theta_VAR}. We abuse the notation $[\bmPhi_1, \cdots, \bmPhi_{p}] \cdot \bfX = \bmPhi_1 \cdot \bfZ_p + \dotsc + \bmPhi_p \cdot \bfZ_{1}$. Due to Assumption \ref{A5:noise} and \eqref{e:VAR(p)}, we know that $  \bbE \abs{\bfY}^{\alpha} < \infty $ and $ \bbE_{(\bfX, \bfY) \sim \bm{\Pi}}\abs{\bfY- [\bmPhi_1, \cdots, \bmPhi_{p}] \cdot \bfX}^\alpha  < \infty $. 
	Our corresponding $\psi_\alpha$ truncated LAD is defined as: 
	\begin{equation}
		\label{C-VAR}
		\begin{aligned}
			[\widehat{\bmPhi}_1, \dotsc, \widehat{\bmPhi}_{p}] &= \argmin_{[\bmPhi_1, \cdots, \bmPhi_{p}] \in \bm{\Theta}} \left\{ \widehat{R}_{\psi_\alpha, \ell_1}(\bmPhi_1, \cdots, \bmPhi_{p})  + \gamma \sum_{i = 1}^p \norm{ \bmPhi_i }_{1,1} \right\}, \\
			\widehat{R}_{\psi_\alpha, \ell_1}(\bmPhi_1,\dotsc,\bmPhi_{p}) &= \frac{1}{ n \lambda}\sum_{i=1}^{n}\psi_\alpha \big( \lambda\abs{\bfZ_{i+p}-(\bmPhi_1 \bfZ_{i+p-1} + \cdots + \bmPhi_p \bfZ_{i})} \big), 
		\end{aligned}
	\end{equation} 
	where $\psi_{\alpha}$ is in \eqref{e:T-func}. To make notations simple, we denote
	\begin{equation}
		\label{e:min Phi}
		\bm{\theta}^* = [ \bmPhi_1^*, \dotsc, \bmPhi_{p}^* ]\quad \text{and} \quad 
		\hat{\bm{\theta}} = [ \widehat{\bmPhi}_1, \dotsc, \widehat{\bmPhi}_p ].
	\end{equation}
	We have the following theoretical result about excess risk, whose proof will be given in \autoref{sec: proof}
	
	\begin{theorem}
		\label{thm:VAR}
		Keep the same conditions as in Lemma \ref{lem:mixing VARp}. For any $\varepsilon >0$, when the sample size $n$ is sufficiently large, there exists some constant $C>0$ independent of $\varepsilon, p, d, \rho$ and $\alpha$ such that 
		\[
		\popR(\hat{\bm{\theta}}) - \popR(\bm{\theta}^*)  \le 
		C\left( \frac{\log n}{  \abs{\log \rho} n } \big( \abs{\log \varepsilon} + p^2 d^2 \log n \big) \right)^{(\alpha-1)/\alpha} , 
		\]
		with probability at least $1-2\varepsilon$. Here, $\bm{\theta}^*$ and $\hat{\bm{\theta}}$ are defined in \eqref{e:min Phi}.
	\end{theorem}	
	
	\section{Simulations}
	\label{sec:simulation}
	
	\subsection{Numerical Studies} 
	
	To illustrate the effectiveness of the $\psi_{\alpha}$ truncation, we consider the following $5$-dimensional VAR$(p)$ models with $p=1$ and $p=2$ respectively:
	\begin{align*}
		\text{ VAR}(1): \bfZ_t &= \diag\{ 0.6, -0.4, 0.1, 0.5, -0.2 \} \bfZ_{t-1} + \bmepsilon_t - \bbE[ \bmepsilon_t]  , \quad \bfZ_t \in \bbR^5 ,  \\
		\text{ VAR}(2): \bfZ_t &= \diag\{ 0.6, -0.6, 0.1, 0.5, -0.2 \}  \bfZ_{t-1}  \\
		&\pheq + \diag\{ -0.3, 0.5, -0.2, -0.3, 0.1 \} \bfZ_{t-2} + \bmepsilon_t  - \bbE [\bmepsilon_t]  , \quad \bfZ_t \in \bbR^5 
	\end{align*}
	where $t \in \mathbb{N}$, and $\{ \bmepsilon_t\}_{t \ge 1}$ is a sequence of  independent random vectors with i.i.d. components following from some heavy-tailed distribution: Pareto distribution with shape $\mu$ (\text{Pareto}($\mu$)) and density function 
	\[
	p_{\mu}(x) = \frac{\mu}{ x^{1+\mu} } \mathbbm{1}_{(1,\infty)}(x), \quad \mu \in (1, 2],
	\] 
	or Fr\'echet distribution with shape $\nu$ (Fr\'echet($\nu$)) and density function
	\[
	p_{\nu}(x) = \frac{\nu}{x^{1+\nu}} \exp\{ -x^{-\nu} \} \mathbbm{1}_{(0,\infty)}(x), \quad \nu \in (1, 2].
	\] 
	It is obvious that VAR(1) and VAR(2) are VAR(1) and VAR (2) models respectively. 
	
	To fit the models in the framework in Section \ref{sec:sec2}, we take 
	\begin{align}
		\bfY_t &= \bfZ_t, \quad \bfX_t = \bfZ_{t-1} \quad {\rm \ in \ VAR(1)}, \label{e:VAR1}\\
		\bfY_t &= \bfZ_t, \quad \bfX_t = (\bfZ_{t-1}, \bfZ_{t-2}) \quad {\rm \  in \ VAR(2)}.  \label{e:VAR2}
	\end{align}
	Besides, we use the following truncation function $\psi_{\alpha}(r)$
	\[
	\psi_{\alpha}(r) = \log \left(1 + r + \frac{ r^{\alpha} }{ \alpha } \right)\ , \quad r  \ge  0,
	\]
	where parameter $\alpha$ satisfies $1 < \alpha < \mu$ for the Pareto($\mu$) distribution or $1 < \alpha < \nu$ for the Fr\'echet($\nu$) distribution, respectively.
	
	Recall that the minimization problem \eqref{e:C-min} with the truncation function $\psi_{\alpha}$ is
	\begin{equation*}
		\hat{\bmtheta}(\lambda, \gamma) := \argmin_{\bmtheta \in \bm{\Theta}}\left\{ \frac{1}{n \lambda} \sum_{i = 1}^n \psi_\alpha( \lambda \abs{ \bfY_i - \bmtheta \cdot \bfX_i } ) + \gamma \norm{ \bmtheta}_{1,1} \right\} .
	\end{equation*}
	where $\lambda$ and $\gamma$ will be selected later. In practice, we can apply stochastic gradient descent (SGD) \cite{MR3948080,MR4065161,gower2019sgd}, to solve this optimization problem. That is,
	\begin{equation}
		\label{eq:SGD catoni}
		\hat{\bmtheta}_{k+1} = \hat{\bmtheta}_k - \eta_k \nabla_{\bmtheta} \left\{ \frac{1}{\lambda} \psi_\alpha\Big(\lambda \big|{ \bfY_k - \hat{\bmtheta}_k \cdot \bfX_k } \big| \Big) + \gamma \Vert \hat{\bmtheta}_k \Vert_{1,1}  \right\}, \quad k = 0, 1, 2, \dotsc 
	\end{equation}
	In the $k$-th step, we generate $(\bfY_k, \bfX_k)$ by VAR$(p)$ mentioned above, and let the learning rate be $\eta_k \in (0,1)$. Analogously, for the minimization problem \eqref{e:EmpMin} without truncation, 
	\begin{equation*}
		\bar{\bmtheta}(\gamma) := \argmin_{\bmtheta \in \bm{\Theta}} \left\{ \frac1n \sum_{i = 1}^n \abs{ \bfY_i - \bmtheta \cdot \bfX_i } + \gamma \norm{ \bmtheta}_{1,1} \right\},
	\end{equation*}
	it can be solved by the following SGD in practice,
	\begin{equation}
		\label{eq:SGD emp}
		\bar{\bmtheta}_{k+1} = \bar{\bmtheta}_k - \eta_k \nabla_{\bmtheta} \left\{ \abs{ \bfY_k- \bar{\bmtheta}_k \cdot \bfX_k } + \gamma \Vert \bar{\bmtheta}_k \Vert_{1,1} \right\}, \quad k = 0 ,1, 2, \dotsc
	\end{equation}
	Furthermore, we also consider the following minimization problem with Huber loss: 
	\begin{equation*}
		\tilde{\bmtheta}(\sigma, \tau, \gamma) := \argmin_{\bmtheta\in \bm{\Theta}} \left\{ \frac{1}{n\sigma} \sum_{i = 1}^n h_{\tau} ( \sigma \abs{ \bfY_i - \bmtheta \cdot \bfX_i } ) + \gamma \norm{\bmtheta}_{1,1}  \right\},
	\end{equation*}
	which can be solved by the following SGD similarly:
	\begin{equation}
		\label{eq:SGD huber}
		\tilde{\bmtheta}_{k+1} = \tilde{\bmtheta}_k - \eta_k \nabla_{\bmtheta} \left\{ \frac{1}{\sigma} h_{\tau} \left(\sigma \big| \bfY_k - \tilde{\bmtheta}_k \cdot \bfX_k \big|  \right) + \gamma \Vert \tilde{\bmtheta}_k \Vert_{1,1} \right\}, \quad k = 0, 1, 2, \dotsc
	\end{equation}
	Here, the Huber function $h_{\tau}(r): \bbR \to \bbR^+$ is defined by
	\[
	h_{\tau}(r) = \left\{
	\begin{aligned}
		r^2 / 2 , \qquad & \quad \abs{r}  \le  \tau, \\
		\tau \abs{r} - \tau^2/2 , & \quad \abs{r} > \tau ,
	\end{aligned}
	\right.
	\]
	and the parameters $\sigma$, $\gamma$ and $\tau$ will be chosen later.
	
	Given a $\bmtheta$ and the observed data $\{(\bfX_i, \bfY_i) \}_{1 \le  i  \le  N}$, we will compute the following empirical risk 
	\[
	R_{\ell_1}^e(\bmtheta) = \frac1N \sum_{i = 1}^N \abs{ \bfY_i - {\bmtheta} \cdot \bfX_i }
	\]
	
	\vskip 3mm 
	
	\subsubsection{General description} Given observed data $\{\bfZ_i\}_{1 \le  i  \le  N+L}$ generating by VAR(1) or VAR(2), we first obtain the estimators $\hat{\bmtheta}$ (with $\psi_{\alpha}$ truncation), $\bar{\bmtheta}$ (without truncation) and $\tilde{\bmtheta}$ (with Huber loss) for the true value $\bmtheta^*$ by using the data set  $\{\bfZ_i\}_{1 \le  i  \le  N}$ to train the three corresponding models. In practice, we shall run the above three SGDs to approximate the these three estimators respectively. For each $\hat{\bmtheta}_k$, $\bar{\bmtheta}_k$ and $\tilde{\bmtheta}_k$ arising from the $k$-th step of SGDs \eqref{eq:SGD catoni}, \eqref{eq:SGD emp} and \eqref{eq:SGD huber} respectively, we let
	\begin{equation} \label{e:HatREst}
		\hat R_{\psi_\alpha,k} = R_{\ell_1}^e(\hat{\bmtheta}_k), \quad \hat R_{{\rm LAD},k} = R_{\ell_1}^e(\bar{\bmtheta}_k) \quad \text{and} \quad \hat R_{{\rm Huber},k} = R_{\ell_1}^e(\tilde{\bmtheta}_k).
	\end{equation}
	
	For the obtained estimator $\hat \bmtheta$, we next use the remain $\{\bfZ_i\}_{N+1 \le  i  \le  N+L}$ to test it. More precisely, we use the model and the obtained estimators $\hat \bmtheta$ to predict $L$ values, denoted as $\{\hat{\bfZ}_i^{pre}\}_{N+1 \le  i  \le  N+L}$. For VAR(1), we let 
	\[
	\hat \bfZ_{N+i}^{pre} = \hat \bmtheta \cdot \hat \bfZ_{N+i-1}^{pre}, \quad \hat \bfX_{N+1}^{pre} = \bfZ_{N}, \quad i=1,\dotsc, L.
	\]
	Analogously, for VAR(2), we let
	\[
	\hat \bfZ_{N+i}^{pre} = \hat \bmtheta \cdot  ( \hat \bfZ_{N+i-1}^{pre},\hat\bfZ_{N+i-2}^{pre} ), \quad ( \hat \bfZ_{N}^{pre},\hat\bfZ_{N-1}^{pre} )=(\bfZ_{N}, \bfZ_{N-1}), \quad i=1,\dotsc, L.
	\]
	We then compute the following prediction error: 
	\[
	\hat E^p(\hat \bmtheta, L) = \frac1L \sum_{i = 1}^L \abs{ \bfZ_{i+N} - \hat \bfZ_{i+N}^{pre} }.
	\]
	The smaller $\hat E^p(\hat \bmtheta, L)$ is, the better the estimator $\hat \bmtheta$ performs. Similarly, we define $\bar{E}^p(\bar{\bmtheta}, L)$ and $\tilde{E}^p(\tilde{\bmtheta}, L)$. 
	For $\hat{\bmtheta}_N$, $\bar{\bmtheta}_N$ and $\tilde{\bmtheta}_N$ arising from SGDs \eqref{eq:SGD catoni}, \eqref{eq:SGD emp} and \eqref{eq:SGD huber} after stopping at the $N$-th iteration respectively, to distinguish among their prediction errors, we let  
	\begin{equation} \label{e:PreError}
		\hat{E}^p_{\psi_{\alpha},L} = \hat E^p(\hat{\bmtheta}_N, L), \quad
		\bar{E}^p_{ {\rm LAD},L} = \bar E^p(\bar{\bmtheta}_N, L), \quad \text{and} \quad
		\tilde{E}^p_{ {\rm Huber},L} = \tilde E^p(\tilde{\bmtheta}_N, L).
	\end{equation} 
	\vskip 3mm
	We will conduct 2000 independent experiments for each of the three models. For each experiment, we run the VAR$(p)$ to create training data for SGDs \eqref{eq:SGD catoni}, \eqref{eq:SGD emp} and \eqref{eq:SGD huber}. To guarantee that the data are approximately stationary, we drop the first $5000$ observations and use the subsequent data $\mathbf{Z}_{5001}, \mathbf{Z}_{5002}, \ldots$. In the $i$-th experiment, $i=1,\dotsc,2000$, let the number of iteration steps of SGDs be $N$. We then obtain $\{\hat{\bmtheta}_k^i, \bar{\bmtheta}_k^i, \tilde{\bmtheta}_k^i : k = 0, \dotsc, N, i=1,\dotsc,2000\}$,  which are used to compute the empirical risk $\hat{R}_{\psi_\alpha,k}$, $\hat{R}_{{\rm LAD},k}$ and $\hat{R}_{{\rm Huber},k}$ as defined in \eqref{e:HatREst}. In \autoref{fig-VAR1p}--\autoref{fig-VAR2f}, the linear plots show the means of these empirical risks. Additionally, we use $\{\hat{\bmtheta}_N^i, \bar{\bmtheta}_N^i, \tilde{\bmtheta}_N^i : i = 1, \dotsc, 2000\}$ to compute the prediction errors $\hat{E}^p_{\psi_{\alpha},L}$, $\hat{E}^p_{{\rm LAD},L}$, and $\hat{E}^p_{{\rm Huber},L}$ as defined in \eqref{e:PreError}. Since the prediction errors may be large, we take the logarithm base $10$ of these values. The means, medians and outliers of these logged prediction errors are then illustrated using the box plots in  \autoref{fig-VAR1p}--\autoref{fig-VAR2f}. 
	
	Furthermore, to demonstrate the advantages of the truncation function $\psi_{\alpha}$, we will select different values of $\alpha$ according to the index parameters $\mu$ and $\nu$ of the Pareto and Fr\'echet noise distributions respectively. The following scenarios will be considered:
	\begin{itemize}
		\item[$(\runum{1})$] For $\mu = \nu = 1.2$, we choose $\alpha = 1.05$, $1.1$, $1.15$ and $1.18$ within  $\psi_\alpha $. 
		\item[$(\runum{2})$] For $\mu = \nu = 1.5$, we choose $\alpha = 1.1$, $1.2$, $1.3$ and $1.4$ within  $\psi_\alpha $.
		\item[$(\runum{3})$] For $\mu = \nu = 1.8$, we choose $\alpha = 1.2$, $1.4$, $1.6$ and $1.7$ within  $\psi_\alpha $.
	\end{itemize}
	
	\subsubsection{Simulations of VAR(1): }
	
	For all SGDs \eqref{eq:SGD catoni}, \eqref{eq:SGD emp} and \eqref{eq:SGD huber}, we set the initial points $\hat{\bmtheta}_0 = \bar{\bmtheta}_0 = \tilde{\bmtheta}_0 = 2\mathbf{I}_5 $ where $\bfI_5$ is the identity matrix in $\bbR^{5\times 5}$, employ a constant learning rate $\eta_k = 0.01$ for $0 \le  k  \le  N = 800$ and set $\gamma = 0.01$. Additionally, we set $\lambda = 0.035 $ for the truncated SGD \eqref{eq:SGD catoni}, and set $\sigma = 1$ and $\tau = 0.5$ in SGD \eqref{eq:SGD huber}. We predict $L=10$ values in each experiment. The outcomes are illustrated in \autoref{fig-VAR1p} for  Pareto($\mu$) noise and in \autoref{fig-VAR1f} for Fr\'echet($\nu$) noise, respectively. 
	
	\begin{figure*}[htpb]
		\centering
		\subfigure[ $\varepsilon_k \sim \text{Pareto}(1.2)$, $\alpha = 1.05$, $1.1$, $1.15$ and $1.18$ within function $\psi_\alpha $. \label{figure 12p}]{
			\begin{minipage}[t]{0.8\linewidth}
				\centering
				\includegraphics[width=\linewidth]{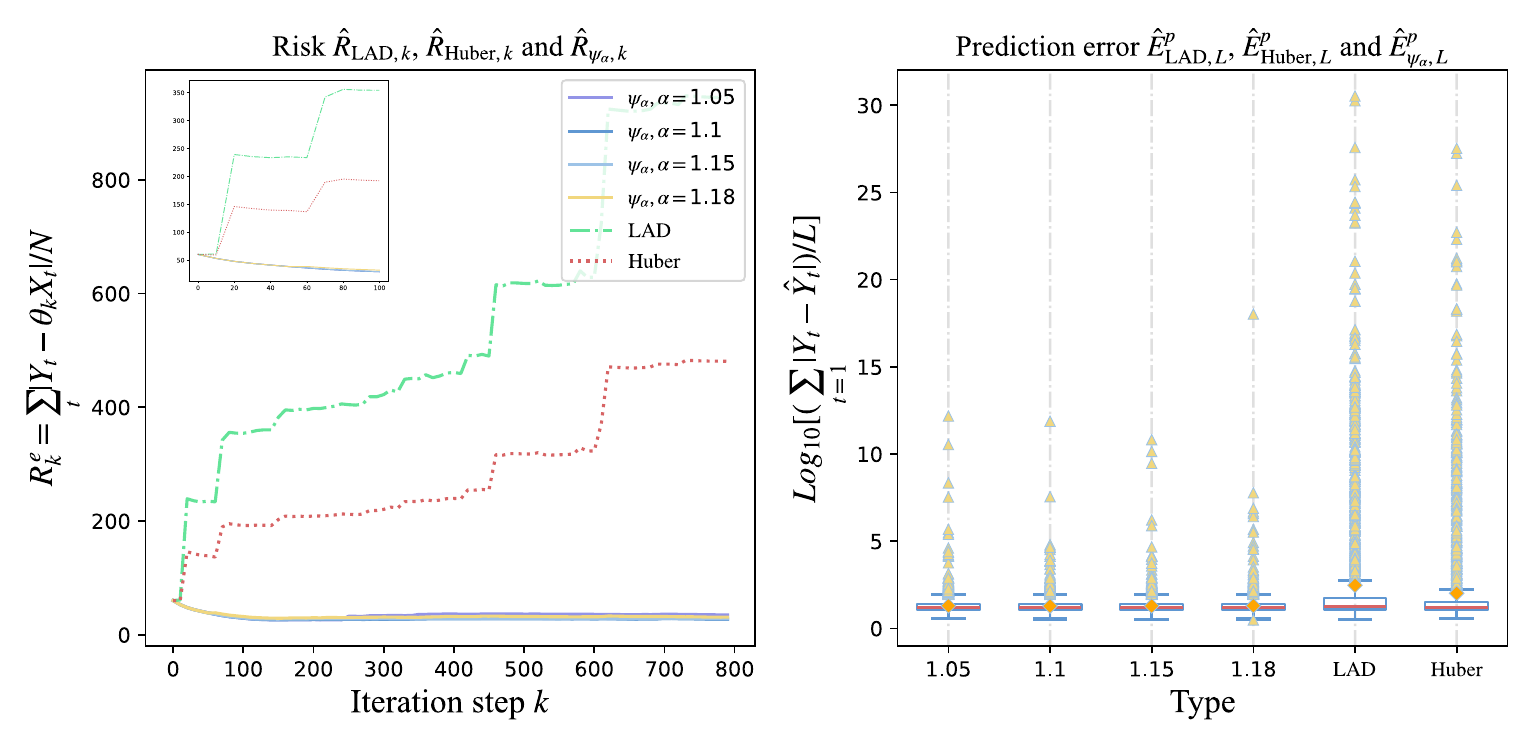 }
			\end{minipage}
		}%
		
		\subfigure[$\varepsilon_k \sim \text{Pareto}(1.5)$, $\alpha = 1.1$, $1.2$, $1.3$ and $1.4$ within  function  $\psi_\alpha $. \label{figure 15p} ]{
			\begin{minipage}[t]{0.8\linewidth}
				\centering
				\includegraphics[width=\linewidth]{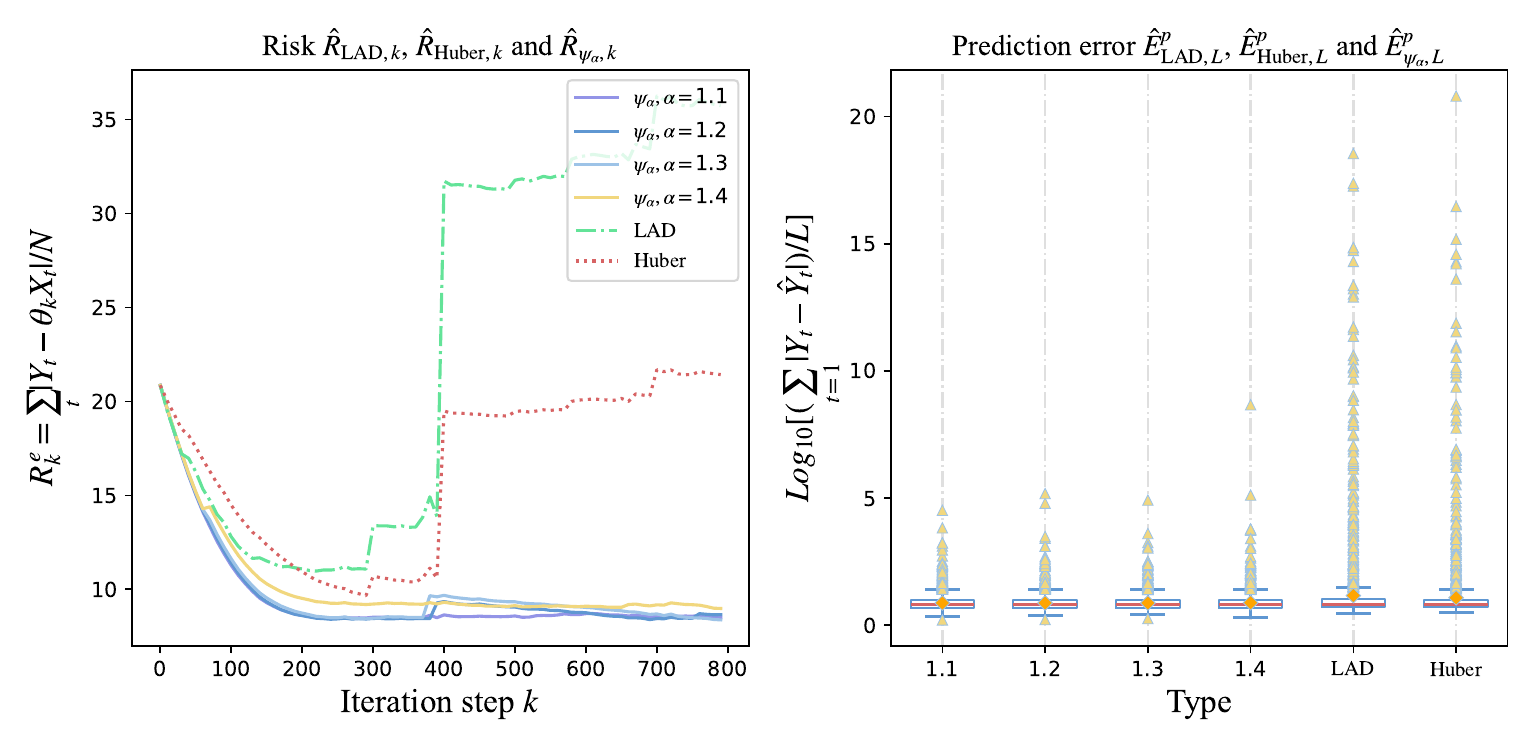}
			\end{minipage}
		}%
		
		\subfigure[ $\varepsilon_k \sim \text{Pareto}(1.8)$, $\alpha = 1.2$, $1.4$, $1.6$ and $1.7$ within  function  $\psi_\alpha $. \label{figure 18p} ]{
			\begin{minipage}[t]{0.8\linewidth}
				\centering
				\includegraphics[width=\linewidth]{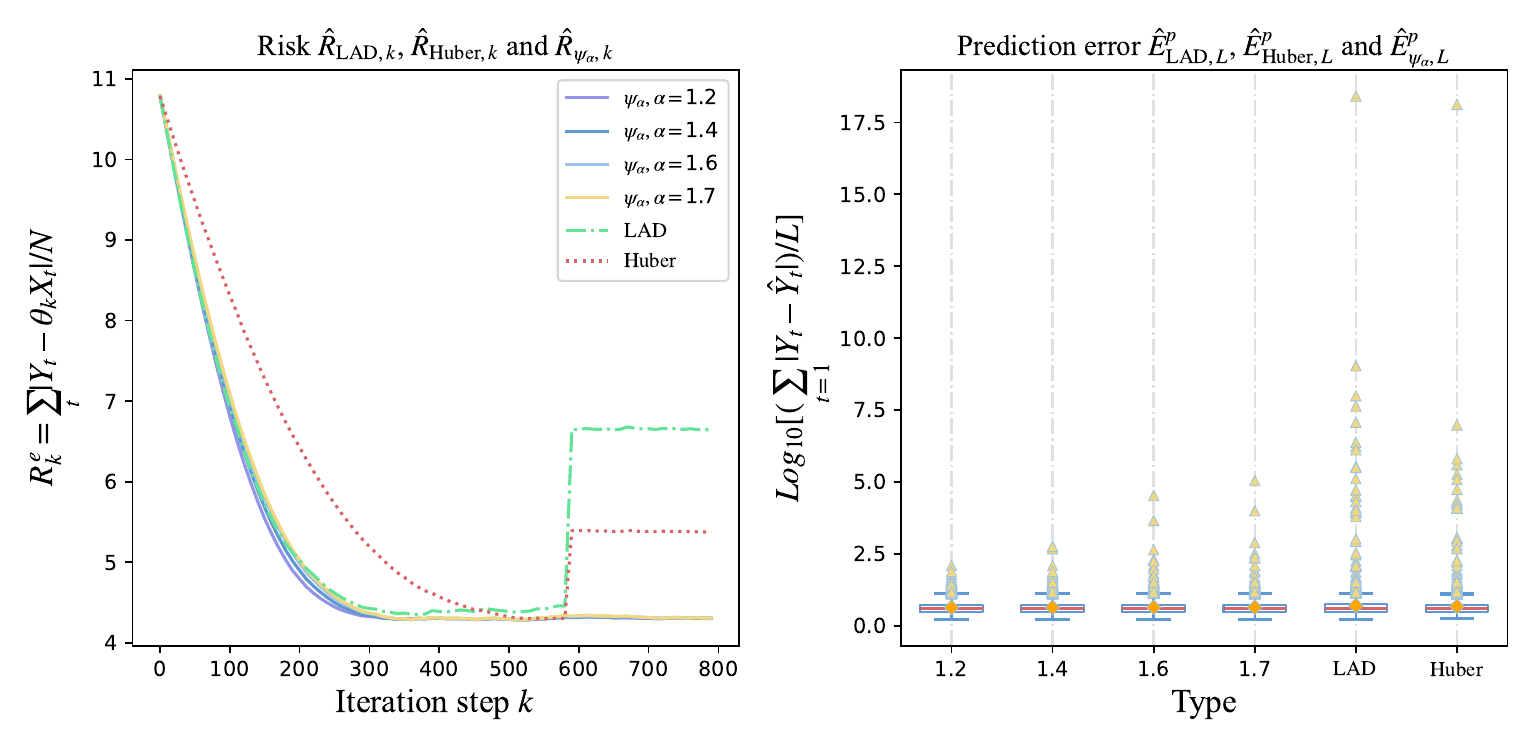}
			\end{minipage}
		}%
		
		\centering
		\caption{  Simulations values of risk $\hat{R}_{\psi_{\alpha},k}$, $\hat{R}_{ {\rm LAD},k}$ and $\hat{R}_{ {\rm Huber},k}$ for VAR(1)  with  $1 \le  k  \le  N = 800$, and the logged prediction errors $\hat{E}^p_{\psi_{\alpha},L}$, $\hat{E}^p_{{\rm LAD},L}$ and $\hat{E}^p_{ {\rm Huber},L}$ with $L=10$. The noise $\varepsilon_k $ follows a Pareto distribution $\text{Pareto}(\mu)$ with $\mu = 1.2$, $1.5$ and $1.8$. }
		\label{fig-VAR1p}
	\end{figure*}
	
	\begin{figure*}[htpb]
		\centering
		\subfigure[$\varepsilon_k \sim \text{Fr\'echet}(1.2)$, $\alpha = 1.05$, $1.1$, $1.15$ and $1.18$ within function $\psi_\alpha $. \label{figure 12invw} ]{
			\begin{minipage}[t]{0.8\linewidth}
				\centering
				\includegraphics[width=\linewidth]{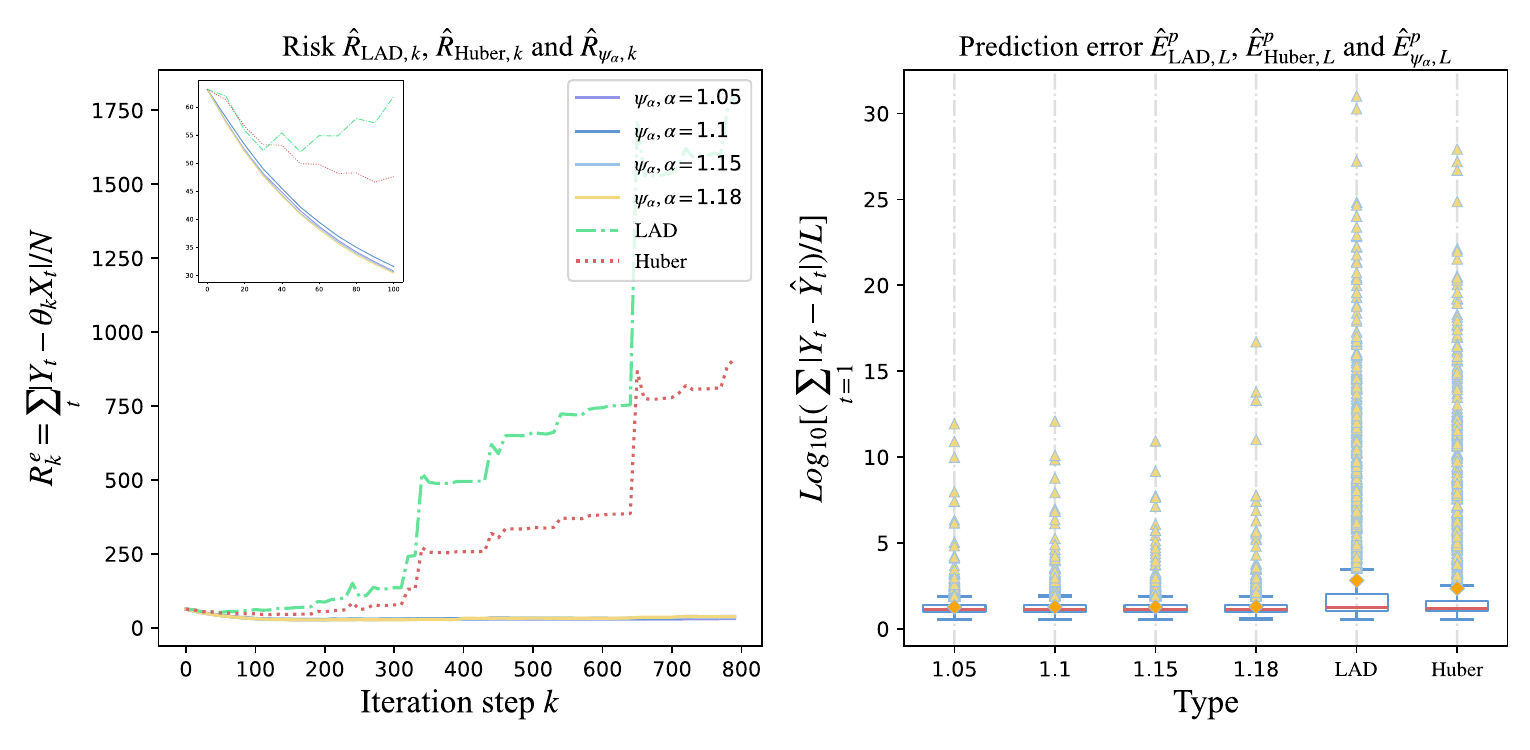}
			\end{minipage}
		}%
		
		\subfigure[$\varepsilon_k \sim \text{Fr\'echet}(1.5)$, $\alpha = 1.1$, $1.2$, $1.3$ and $1.4$ within  function $\psi_\alpha $. \label{figure 15invw} ]{
			\begin{minipage}[t]{0.8\linewidth}
				\centering
				\includegraphics[width=\linewidth]{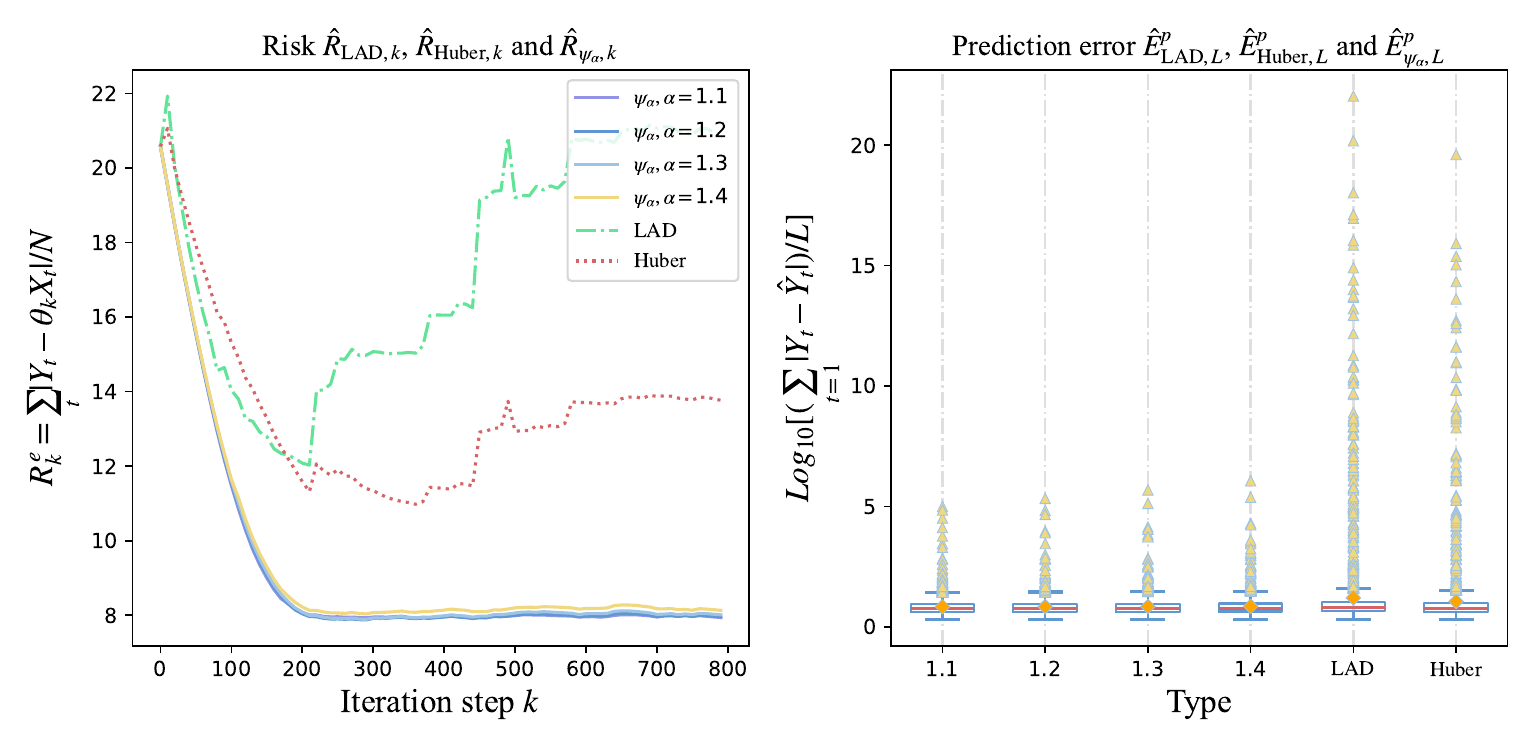}
			\end{minipage}
		}%
		
		\subfigure[$\varepsilon_k \sim \text{Fr\'echet}(1.8)$, $\alpha = 1.2$, $1.4$, $1.6$ and $1.7$ within function $\psi_\alpha $. \label{figure 18invw}  ]{
			\begin{minipage}[t]{0.8\linewidth}
				\centering
				\includegraphics[width=\linewidth]{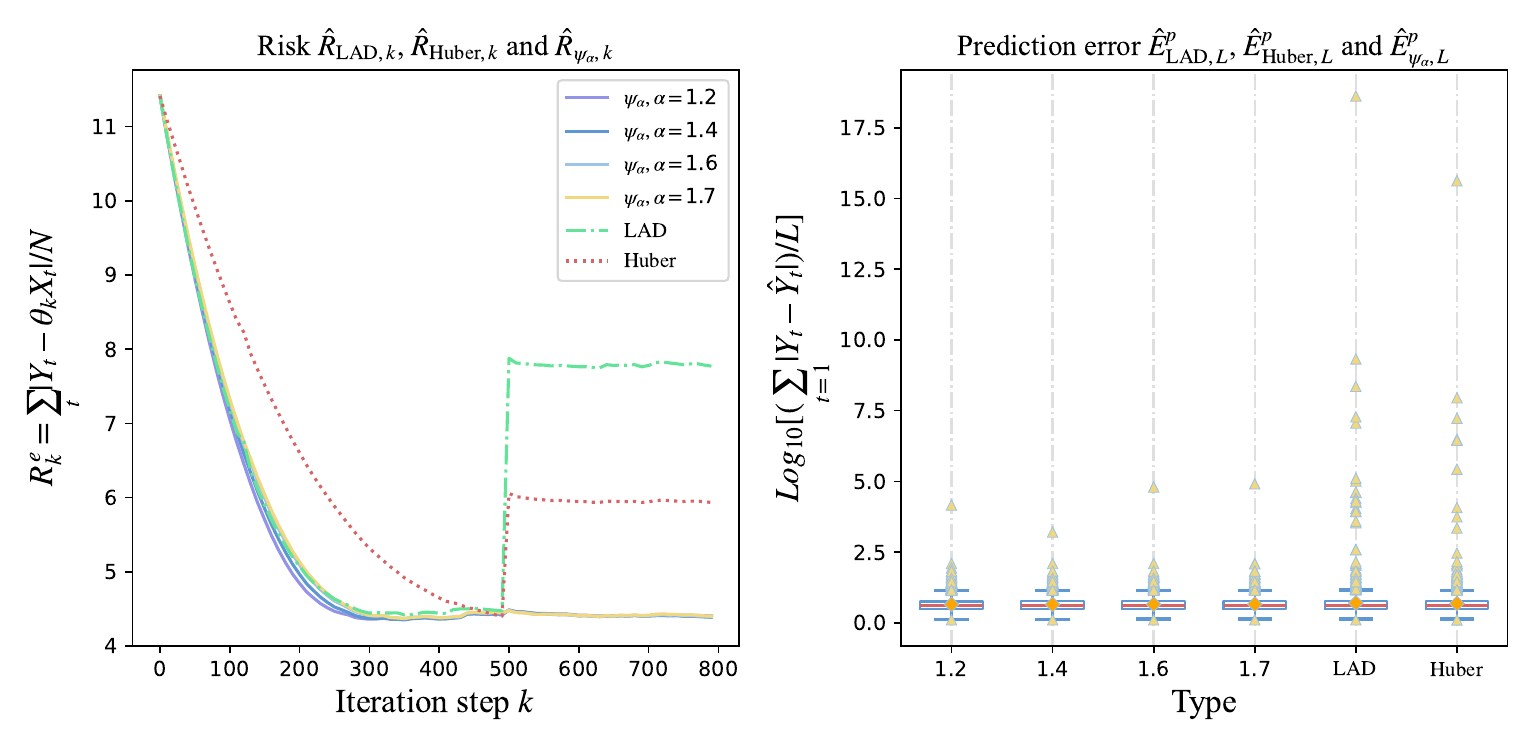}
			\end{minipage}
		}%
		
		\centering
		\caption{   Simulations values of risk $\hat{R}_{\psi_{\alpha},k}$, $\hat{R}_{ {\rm LAD},k}$ and $\hat{R}_{ {\rm Huber},k}$ for VAR(1)  with  $1 \le  k  \le  N = 800$, and the logged prediction errors $\hat{E}^p_{\psi_{\alpha},L}$, $\hat{E}^p_{{\rm LAD},L}$ and $\hat{E}^p_{ {\rm Huber},L}$ with $L=10$. The noise $\varepsilon_k $ follows a Fr\'echet distribution $\text{Fr\'echet}(\nu)$ with $\nu=1.2$, $1.5$ and $1.8$. }
		\label{fig-VAR1f}
	\end{figure*}
	
	\subsubsection{Simulations of VAR(2): }
	
	For all SGDs \eqref{eq:SGD catoni}, \eqref{eq:SGD emp} and \eqref{eq:SGD huber}, the initial points are set to $\hat{\bmtheta}_0 = \bar{\bmtheta}_0 = \tilde{\bmtheta}_0 =  (\bfI_5, \bfI_5) \in \bbR^{5 \times 10}$, a constant learning rate of $\eta_k \equiv 0.01$ is used, the number of iterations is $N=800$, and the penalty coefficient is $\gamma = 0.005$. Additionally, we set $\lambda = 0.035$ in SGD \eqref{eq:SGD catoni}, let $\sigma = 1$ and $\tau = 0.5$ in SGD \eqref{eq:SGD huber}. We predict $L=10$ values in each experiment. Consequently, we obtain \autoref{fig-VAR2p} for Pareto($\mu$) noise and in \autoref{fig-VAR2f} for Fr\'echet($\nu$) noise, respectively.		
	
	\begin{figure*}[htpb]
		\centering
		\subfigure[ $\varepsilon_k \sim \text{Pareto}(1.2)$, $\alpha = 1.05$, $1.1$, $1.15$ and $1.18$ within function $\psi_\alpha $. \label{figure 2_12p}]{
			\begin{minipage}[t]{0.8\linewidth}
				\centering
				\includegraphics[width=\linewidth]{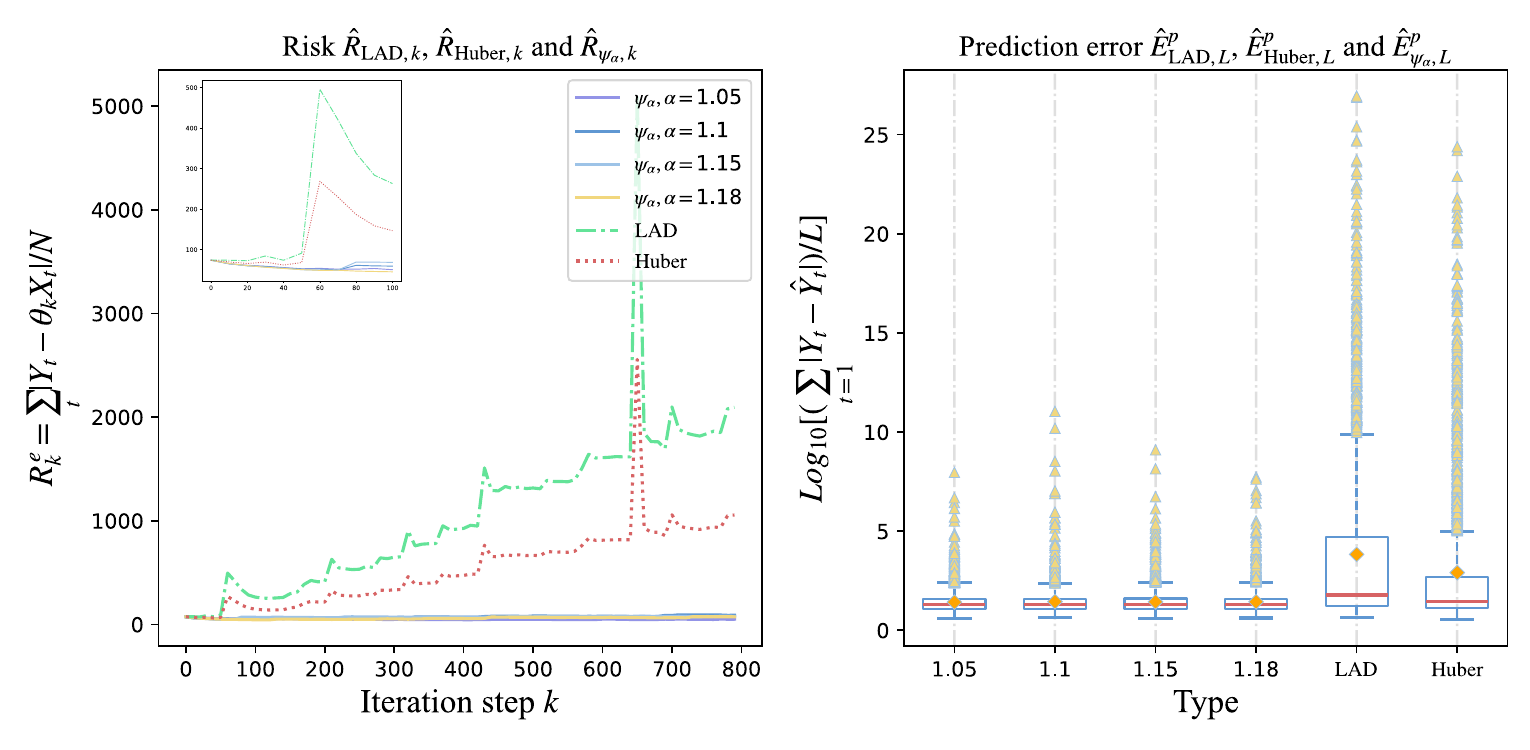}
			\end{minipage}
		}%
		
		\subfigure[$\varepsilon_k \sim \text{Pareto}(1.5)$, $\alpha = 1.1$, $1.2$, $1.3$ and $1.4$ within function $\psi_\alpha $. \label{figure 2_15p} ]{
			\begin{minipage}[t]{0.8\linewidth}
				\centering
				\includegraphics[width=\linewidth]{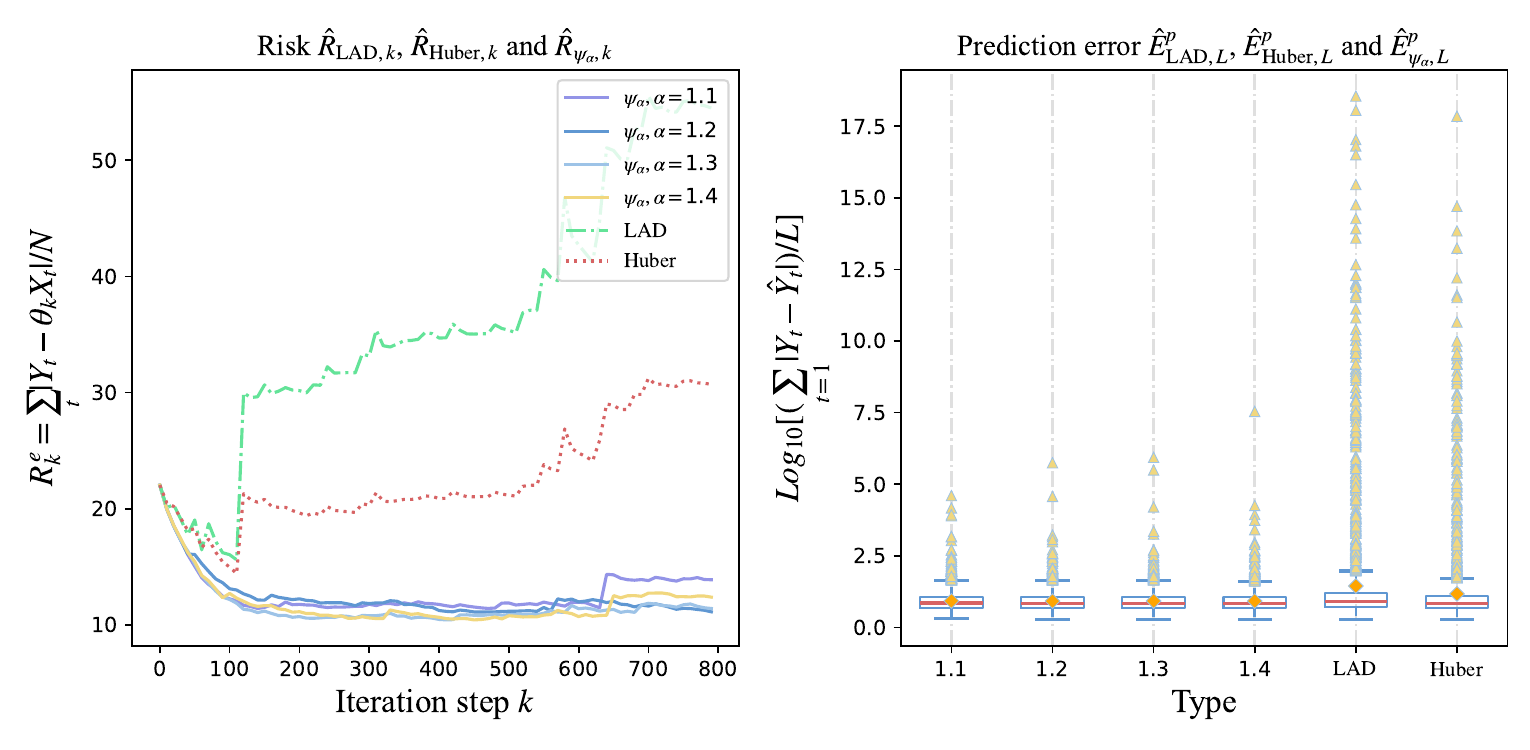}
			\end{minipage}
		}%
		
		\subfigure[ $\varepsilon_k \sim \text{Pareto}(1.8)$, $\alpha = 1.2$, $1.4$, $1.6$ and $1.7$ within function $\psi_\alpha $. \label{figure 2_18p} ]{
			\begin{minipage}[t]{0.8\linewidth}
				\centering
				\includegraphics[width=\linewidth]{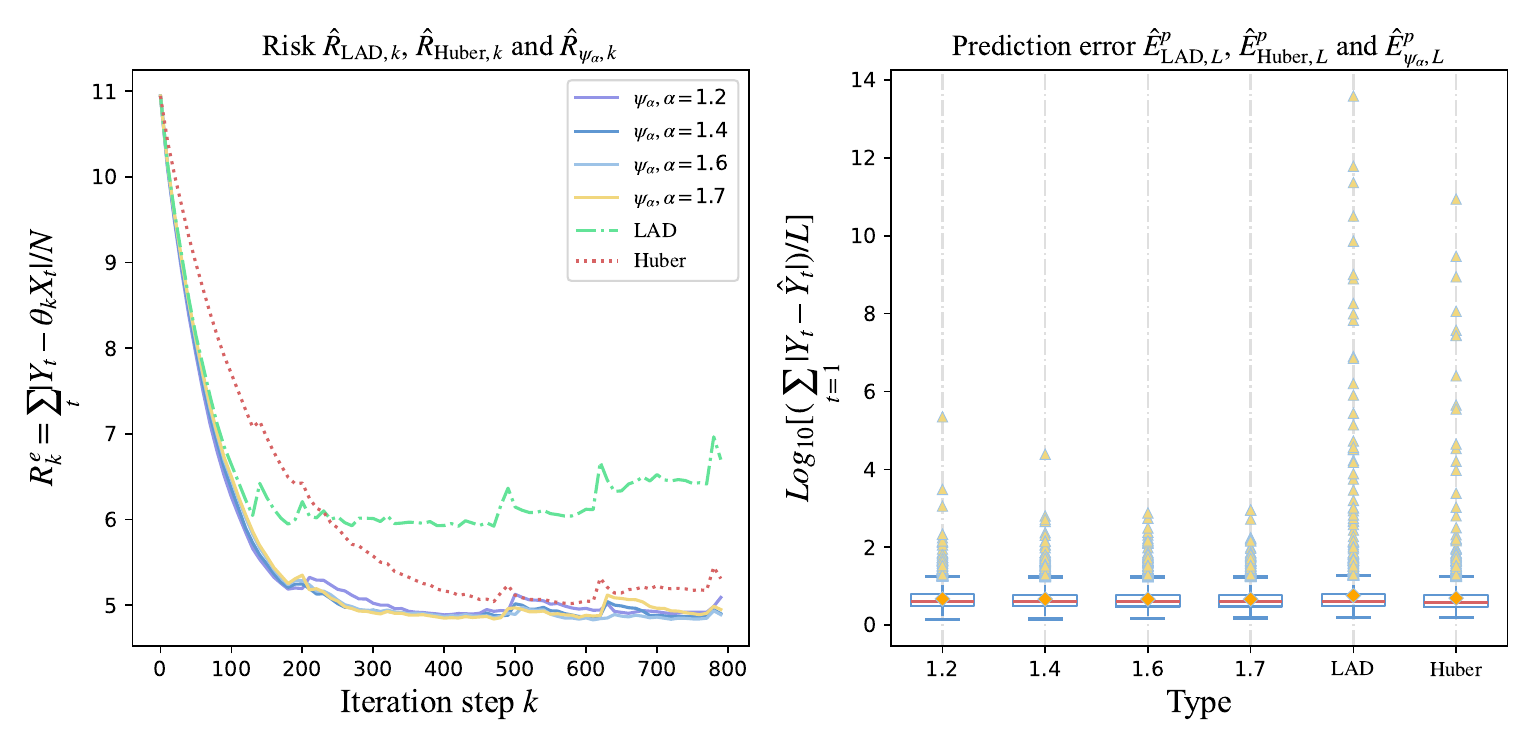}
			\end{minipage}
		}%
		
		\centering
		\caption{ Simulations values of risk $\hat{R}_{\psi_{\alpha},k}$, $\hat{R}_{ {\rm LAD},k}$ and $\hat{R}_{ {\rm Huber},k}$ for VAR(2)  with  $1 \le  k  \le  N = 800$, and the logged prediction errors $\hat{E}^p_{\psi_{\alpha},L}$, $\hat{E}^p_{{\rm LAD},L}$ and $\hat{E}^p_{ {\rm Huber},L}$ with $L=10$. The noise $\varepsilon_k $ follows a Pareto distribution $\text{Pareto}(\mu)$ with $\mu = 1.2$, $1.5$ and $1.8$. }
		\label{fig-VAR2p}
	\end{figure*}
	
	\begin{figure*}[htpb]
		\centering
		\subfigure[ $\varepsilon_k \sim \text{Fr\'echet}(1.2)$, $\alpha = 1.05$, $1.1$, $1.15$ and $1.18$ within function $\psi_\alpha $. \label{figure 12invw_2} ]{
			\begin{minipage}[t]{0.8\linewidth}
				\centering
				\includegraphics[width=\linewidth]{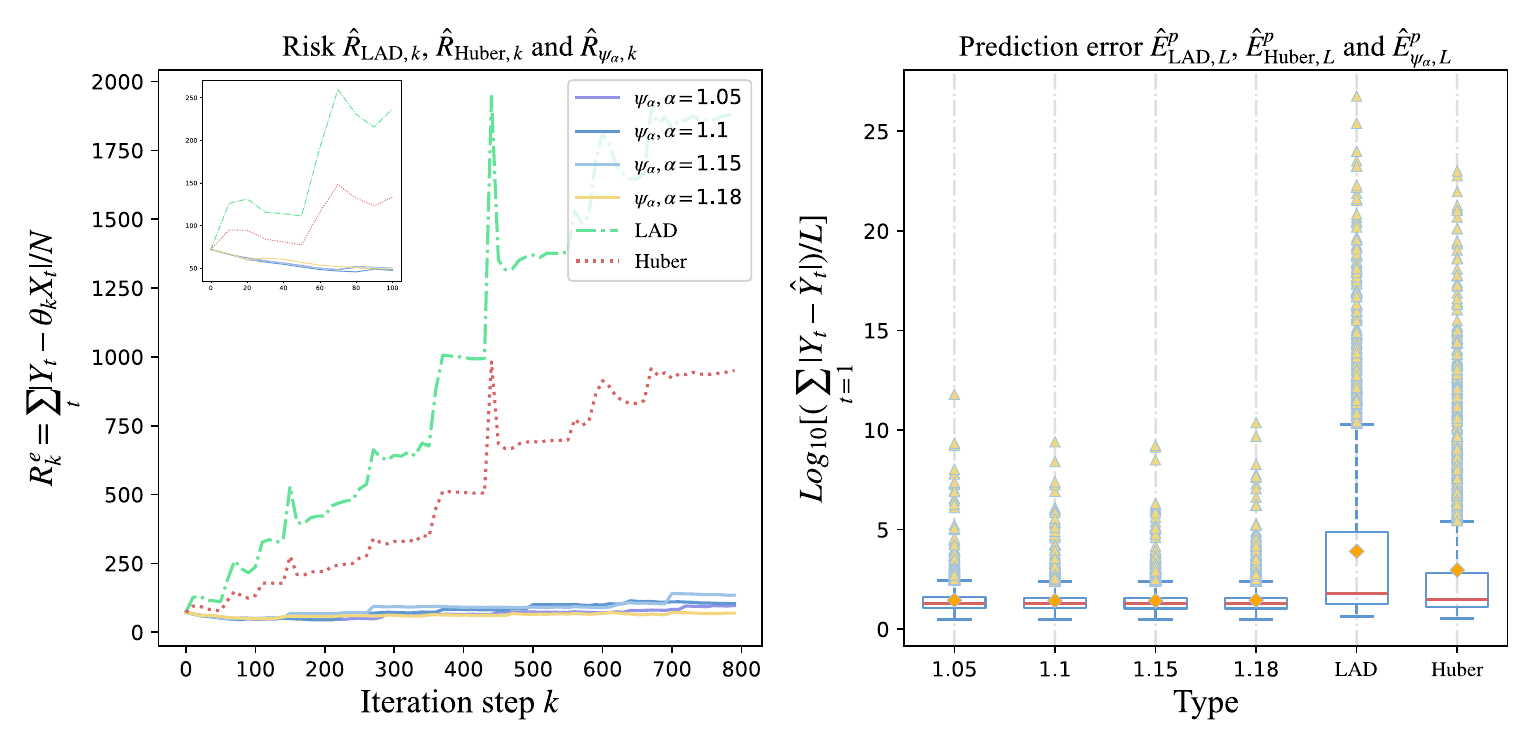}
			\end{minipage}
		}%
		
		\subfigure[ $\varepsilon_k \sim \text{Fr\'echet}(1.5)$, $\alpha = 1.1$, $1.2$, $1.3$ and $1.4$ within function $\psi_\alpha $. \label{figure 15invw_2} ]{
			\begin{minipage}[t]{0.8\linewidth}
				\centering
				\includegraphics[width=\linewidth]{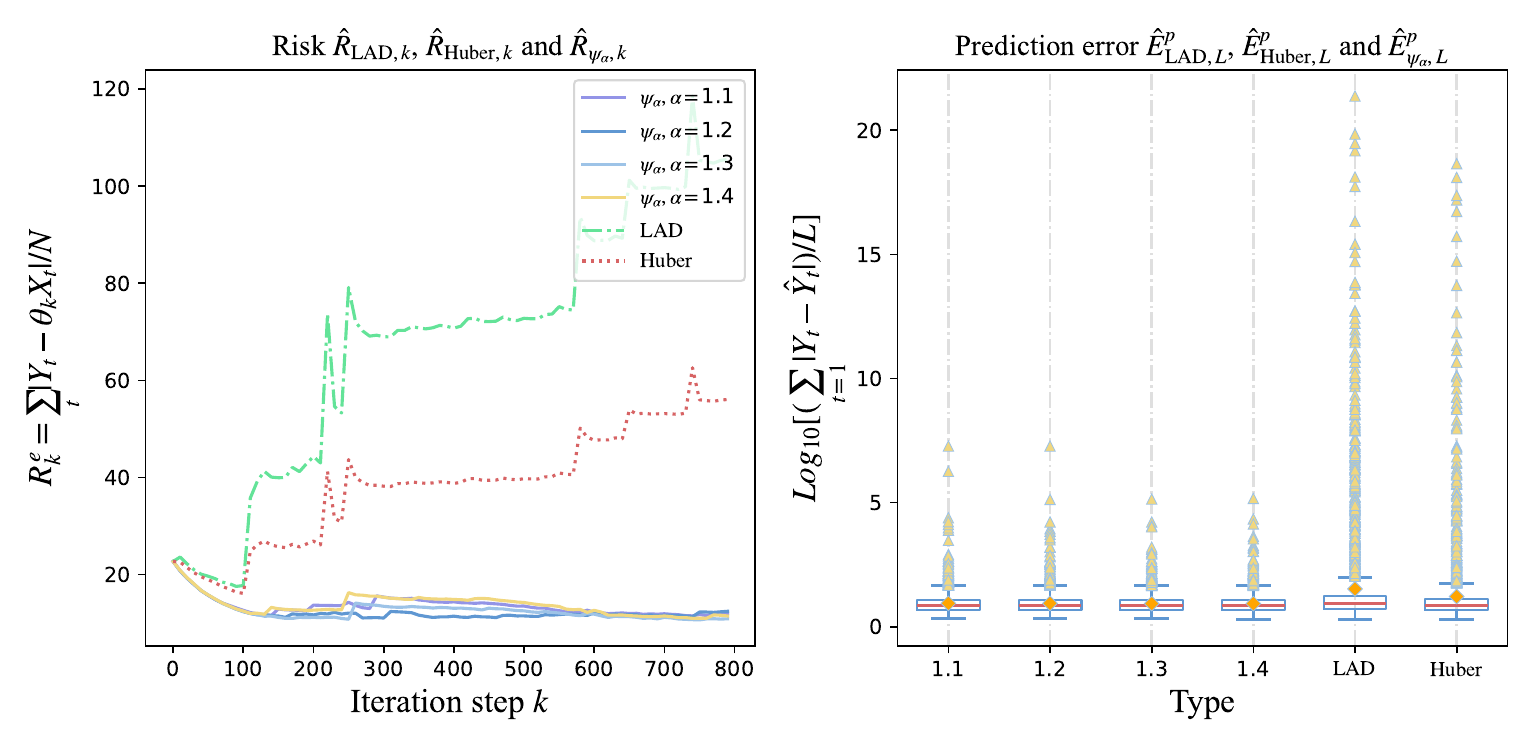}
			\end{minipage}
		}%
		
		\subfigure[ $\varepsilon_k \sim \text{Fr\'echet}(1.8)$, $\alpha = 1.2$, $1.4$, $1.6$ and $1.7$ within function $\psi_\alpha $. \label{figure 18invw_2} ]{
			\begin{minipage}[t]{0.8\linewidth}
				\centering
				\includegraphics[width=\linewidth]{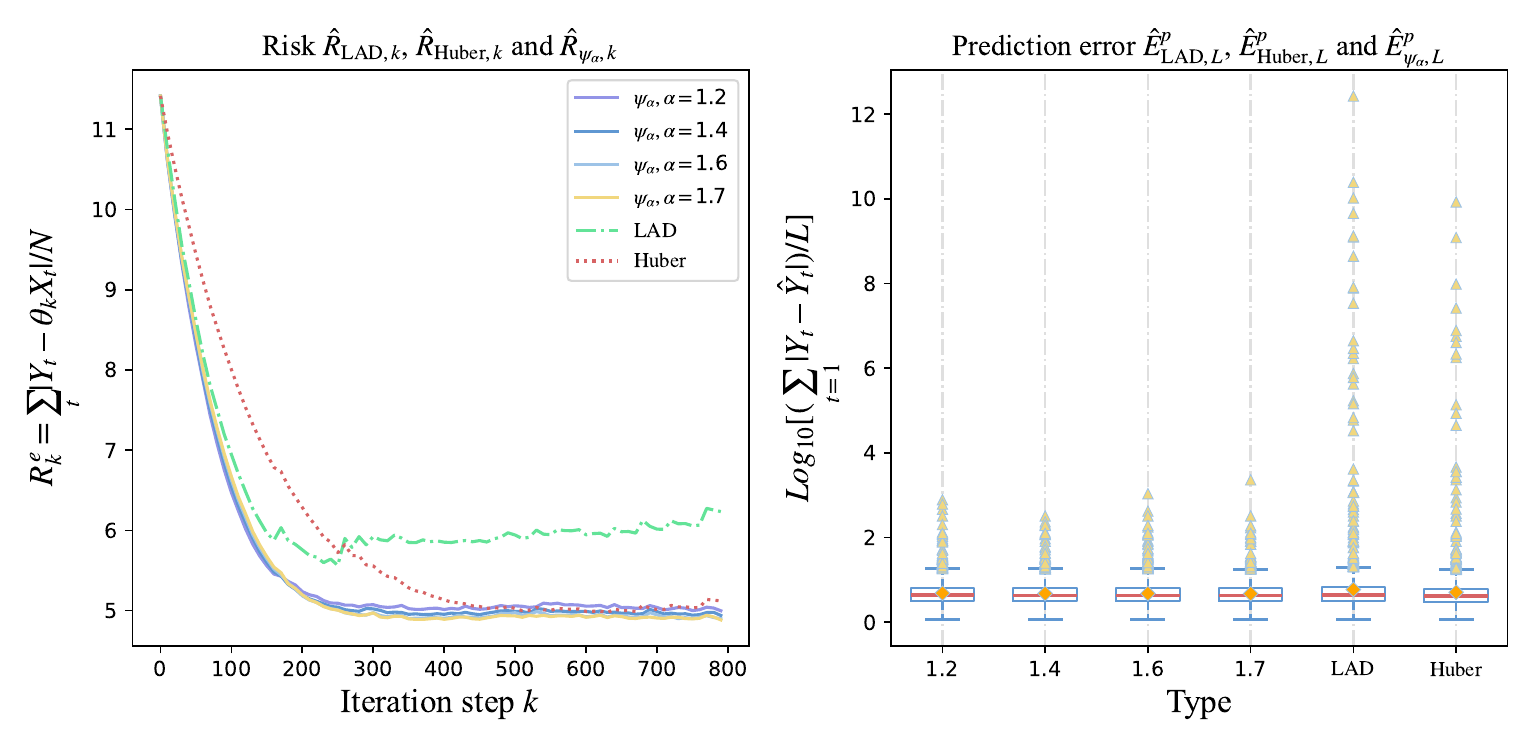}
			\end{minipage}
		}%
		
		\centering
		\caption{  Simulations values of risk $\hat{R}_{\psi_{\alpha},k}$, $\hat{R}_{ {\rm LAD},k}$ and $\hat{R}_{ {\rm Huber},k}$ for VAR(2)  with  $1 \le  k  \le  N = 800$, and the logged prediction errors $\hat{E}^p_{\psi_{\alpha},L}$, $\hat{E}^p_{{\rm LAD},L}$ and $\hat{E}^p_{ {\rm Huber},L}$ with $L=10$. The noise $\varepsilon_k $ follows a Fr\'echet distribution $\text{Fr\'echet}(\nu)$ with $\nu=1.2$, $1.5$ and $1.8$. }
		\label{fig-VAR2f}
	\end{figure*}

	\subsubsection{Conclusion:}
	
	Besides the figures, we also present \autoref{tab:pareto} for Pareto($\mu$) noise and \autoref{tab:invweibull} for Fr\'echet($\nu$) noise, which compare the differences with respect to the order $p$ more directly.
	\begin{table}[htbp]
		\centering
		\caption{ Comparison of average empirical $\ell_1$ risk  and logged prediction error for LAD regression on VAR(1) and VAR(2) models with Pareto($\mu$) noise, where $\mu = 1.2$, $1.5$, and $1.8$. }
		\label{tab:pareto}%
		\renewcommand\arraystretch{1.13}
		\resizebox{1.0\linewidth}{!}{
			\begin{tabular}{c|c|c|cc|cc}
				\bottomrule
				&   \multicolumn{1}{c}{}    &       & \multicolumn{2}{c|}{VAR(1)} & \multicolumn{2}{c}{VAR(2)} \bigstrut\\
				\hline
				& \multicolumn{2}{c|}{\textbf{Pareto($\mu$)}} & \multirow{1}[4]{*}{Risk } & \multicolumn{1}{c|}{\multirow{1}[4]{*}{Prediction error }} & \multirow{1}[4]{*}{Risk } & \multicolumn{1}{c}{\multirow{1}[4]{*}{Prediction error}} \bigstrut\\
				\cline{2-3}          & $\mu$    & $\alpha$ &   (mean)    &     (logged mean )   &  (mean)     &  (logged mean ) \bigstrut\\
				\hline
				\hline
				\multirow{12}[24]{*}{ \rotatebox{90}{\textbf{ $\psi_\alpha$ Truncation}} } & \multirow{4}[8]{*}{1.20 } & 1.05  & 34.974  & 1.290  & 50.334  & 1.422  \bigstrut\\
				&       & 1.10  & 27.360  & 1.280  & 88.064  & 1.447  \bigstrut\\
				&       & 1.15  & 28.091  & 1.285  & 77.914  & 1.437 \bigstrut\\
				&       & 1.18  & 30.930  & 1.299  & 73.195  & 1.436 \bigstrut\\
				\cline{2-7}          & \multirow{4}[8]{*}{1.50 } & 1.10  & 8.498  & 0.877  & 13.896  & 0.928 \bigstrut\\
				&       & 1.20  & 8.653  & 0.879  & 11.136  & 0.923  \bigstrut\\
				&       & 1.30  & 8.373  & 0.877  & 11.423  & 0.924 \bigstrut\\
				&       & 1.40  & 8.982  & 0.886  & 12.400  & 0.923 \bigstrut\\
				\cline{2-7}          & \multirow{4}[8]{*}{1.80 } & 1.20  & 4.303  & 0.637  & 5.090  & 0.673  \bigstrut\\
				&       & 1.40  & 4.305  & 0.640  & 4.900  & 0.668  \bigstrut\\
				&       & 1.60  & 4.309  & 0.644  & 4.889  & 0.665   \bigstrut\\
				&       & 1.70  & 4.309  & 0.645  & 4.948  & 0.663  \bigstrut\\
				\hline
				\hline
				\multirow{3}[6]{*}{ \rotatebox{90}{\textbf{LAD }} } & 1.20  & \multirow{3}[6]{*}{*} & 908.465  & 2.472  & 2094.062  & 3.830   \bigstrut\\
				& 1.50  &       & 35.732  & 1.168  & 54.492  & 1.447  \bigstrut\\
				& 1.80  &       & 6.648  & 0.702  & 6.698  & 0.760  \bigstrut\\
				\hline
				\hline
				\multirow{3}[6]{*}{ \rotatebox{90}{\textbf{Huber}} } & 1.20  & \multirow{3}[6]{*}{*} & 480.862  & 2.015   & 1055.114  & 2.912 \bigstrut\\
				& 1.50  &       & 21.425  & 1.083  & 30.705  & 1.173  \bigstrut\\
				& 1.80  &       & 5.378  & 0.678  & 5.310  & 0.693  \bigstrut\\
				\toprule
		\end{tabular} }%
	\end{table}%
	
	\begin{table}[htbp]
		\centering
		\caption{Comparison of average empirical $\ell_1$ risk and logged prediction error for LAD regression on VAR(1) and VAR(2) models with Fr\'echet($\nu$) noise, where $\nu = 1.2$, $1.5$, and $1.8$.}
		\label{tab:invweibull}%
		\renewcommand\arraystretch{1.13}
		\resizebox{1.0\linewidth}{!}{
			\begin{tabular}{c|c|c|cc|cc}
				\bottomrule
				&   \multicolumn{1}{c}{}    &       & \multicolumn{2}{c|}{VAR(1)} & \multicolumn{2}{c}{VAR(2)} \bigstrut\\
				\hline
				& \multicolumn{2}{c|}{\textbf{Fr\'echet($\nu$)}} & \multirow{1}[4]{*}{Risk } & \multicolumn{1}{c|}{\multirow{1}[4]{*}{Prediction error }} & \multirow{1}[4]{*}{Risk } & \multicolumn{1}{c}{\multirow{1}[4]{*}{Prediction error}} \bigstrut\\
				\cline{2-3}          & $\nu$    & $\alpha$ &   (mean)    &     (logged mean )   &  (mean)     &  (logged mean ) \bigstrut\\
				\hline
				\hline
				\multirow{12}[24]{*}{ \rotatebox{90}{\textbf{ $\psi_\alpha$ Truncation }} } & \multirow{4}[8]{*}{1.20 } & 1.050  & 30.986  & 1.271  & 96.804  & 1.454 \bigstrut[t]\\
				&       & 1.10  & 37.940  & 1.281  & 103.261  & 1.428  \\
				&       & 1.15  & 34.453  & 1.273  & 134.532  & 1.429  \\
				&       & 1.18  & 37.656  & 1.293  & 69.266  & 1.449   \bigstrut[b]\\
				\cline{2-7}          & \multirow{4}[2]{*}{1.50 } & 1.10  & 7.942  & 0.841  & 12.024  & 0.944  \bigstrut[t]\\
				&       & 1.20  & 7.962  & 0.842  & 12.429  & 0.942  \\
				&       & 1.30  & 8.013  & 0.847  & 10.880  & 0.934  \\
				&       & 1.40  & 8.132  & 0.855  & 11.463  & 0.933 \bigstrut[b]\\
				\cline{2-7}          & \multirow{4}[2]{*}{1.80 } & 1.20  & 4.404  & 0.652  & 4.998  & 0.686  \bigstrut[t]\\
				&       & 1.40  & 4.385  & 0.653  & 4.939  & 0.679  \\
				&       & 1.60  & 4.391  & 0.657  & 4.898  & 0.678  \\
				&       & 1.70  & 4.400  & 0.660  & 4.884  & 0.675  \bigstrut[b]\\
				\hline
				\hline
				\multirow{3}[6]{*}{ \rotatebox{90}{\textbf{ LAD }} } & 1.20  & \multirow{3}[6]{*}{*} & 1798.187  & 2.818   & 1888.108  & 3.908   \bigstrut[t]\\
				& 1.50  &       & 20.887  & 1.207  & 105.517  & 1.525  \\
				& 1.80  &       & 7.771  & 0.703  & 6.234  & 0.772  \bigstrut[b]\\
				\hline
				\hline
				\multirow{3}[6]{*}{ \rotatebox{90}{\textbf{Huber}} } & 1.20  & \multirow{3}[6]{*}{*} & 908.080  & 2.350  & 950.772  & 2.969   \bigstrut[t]\\
				& 1.50  &       & 13.770  & 1.052  & 56.051  & 1.219  \\
				& 1.80  &       & 5.935  & 0.684  & 5.121  & 0.707 \bigstrut[b]\\
				\toprule
		\end{tabular} }%
	\end{table}%
	
	Based on the figures and tables obtained, we can find that:
	\begin{itemize}
		\item[$(\runum{1})$] SGD \eqref{eq:SGD catoni} with $\psi_\alpha$ truncation is more efficient and stable than SGD \eqref{eq:SGD emp} without truncation and SGD \eqref{eq:SGD huber} with the Huber loss.
		\item[$(\runum{2})$] The truncation methods are effective in reducing outliers, and $\psi_{\alpha}$ truncation outperforms the Huber one.
		\item[$(\runum{3})$] The $\psi_\alpha$  truncation method yields the smallest means of both empirical risks and prediction errors.
		\item[$(\runum{4})$] For higher dimensions $d$ and orders $p$, the advantages of the $\psi_\alpha$  truncation method become more pronounced.
	\end{itemize}
	
	It is evident that the minimization problem \eqref{e:C-min} with truncation outperforms the non-truncated approach \eqref{e:EmpMin} in the presence of heavy-tailed distributions. Specifically, $\psi_\alpha$ truncation is more effective than the Huber loss. This improvement is particularly notable in cases where the tail distribution is extremely heavy, such as Pareto($ 1.2$) and Fréchet($1.2$). The robustness introduced by the truncation $\psi_\alpha$ effectively mitigates the influence of outliers.
	
	\subsection{Real Data Analysis}
	
	We apply our $\psi_\alpha$ truncation method to the GDP and its decomposition of the United States from 1959 to 2008 \cite{Koop2013, MR4480716}. The total sample size is $ \widetilde{N} = 200$. To make the data stationary, we transform them by the methods described in \cite{MR4480716} and then denote the $t$-th sample by $\bfZ_t$, $t = 1, \dotsc, \widetilde{N}$, where $\bfZ_t \in \bbR^{75}$ for each $t$. According to different transformation methods, we divide $\bfZ_t$ into two parts. That is,  
	\[
	\bfZ_t = ( \bfZ_t^1 , \bfZ_t^2 ), \quad \bfZ_t^1 \in \bbR^{20} \  \text{ and } \  \bfZ_t^2 \in \bbR^{55},
	\]	 
	where the components of $\bfZ_t^1$ are from GDP251 to GDP271, and the components of $\bfZ_t^2$ are from GDP272 to GDP288A. We assume that these data follow a VAR($1$) model. 
	
	To find the tail index, we use the Hill estimator (see, for instance, \cite{MR1939710,MR1482290,MR4174389}). Due to the small sample size $\widetilde{N}$, which limits accurate tail index estimation, we apply the method described in \cite{MR1939710}. For traditional Hill estimators, denoted as $\gamma(k)$, $k=1, \dotsc, \widetilde{N}$, we find that $\gamma(k)$ is approximately linear for $k \leq 160$ according to the first subfigure of \autoref{fig:hill}. We then compute the weighted average of Hill estimators (see, \cite{MR1939710}), denoted as $\gamma^*(\bar{k})$ for $\bar{k}  \le  160$ , resulting in the second subfigure of \autoref{fig:hill}. Consequently, we find $1/\gamma^*(160) \approx 4$, implying that $\bfZ_t$ (or $\bfZ_t^1$, $\bfZ_t^2$) has an $\alpha = 4$-th absolute moment. 
	
	\begin{figure*}[htpb]
		\centering
		\includegraphics[width=0.8\linewidth]{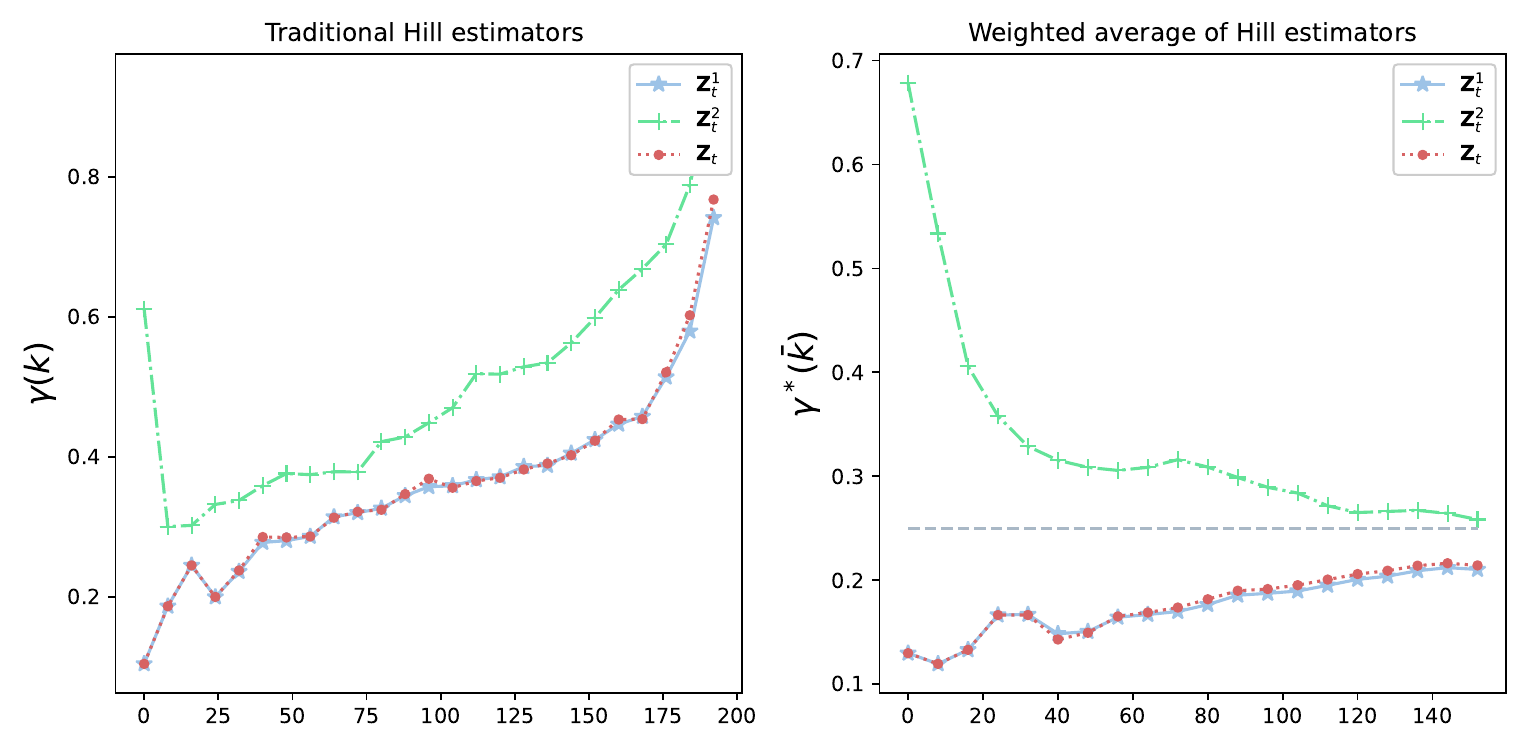}
		\caption{ (Average) Hill estimators for $\bfZ_t^1$, $\bfZ_t^2$ and $\bfZ_t$  }
		\label{fig:hill}
	\end{figure*}
	
	We solve the corresponding minimization problem using SGDs \eqref{eq:SGD catoni}, \eqref{eq:SGD emp} and \eqref{eq:SGD huber}, respectively. For all SGDs, we set the number of iterations to $n = 10000$ and employ a constant learning rate $\eta \equiv 0.08$. We train on $ \bfZ_t$ (or $\bfZ_t^1$, $\bfZ_t^2$) for $1  \le  t  \le  N = 190$. In the $k$-th step of the SGDs, we randomly choose a $t_k \in \{1, \dotsc, 190\}$, and use $\bfZ_{t_k}$ (or $\bfZ_{t_k}^1$, $\bfZ_{t_k}^2$) to obtain $\hat{\bmtheta}_k$, $\bar{\bmtheta}_k$ and $\tilde{\bmtheta}_k$ by SGDs \eqref{eq:SGD catoni}, \eqref{eq:SGD emp} and \eqref{eq:SGD huber} respectively. Finally, we select $\hat{\bmtheta}_k$, $\bar{\bmtheta}_k$ and $\tilde{\bmtheta}_k$ for $k \in \{1000, 1500, \dotsc, 10000\}$ to predict $L = 10$ values and obtain the corresponding prediction errors.  We set the coefficient of the penalty term as $\gamma = 0.5$ to ensure fast convergence of the SGDs, but this may lead to large biases. Therefore, when the rate of change of the empirical  risk falls below a chosen threshold $c_0$, i.e., $\Delta_{k+1} := \abs{R_{\ell_1}^e(\bmtheta_{k+1}) - R_{\ell_1}^e(\bmtheta_k)} / R(\bmtheta_k)  \le  c_0$, we set $\gamma = 0$ to counteract the bias effect of the penalty term. On the other hand, we also compute the rate of change of prediction errors with respect to $\hat{E}^p_{{\rm LAD}, L}$ and $\hat{E}^p_{\psi_{\alpha}, L}$, denoted as $\Delta( {\rm LAD}, \psi_{\alpha})$, given by
	\[
	\Delta( {\rm LAD}, \psi_{\alpha}) := \frac{\hat{E}^p_{ {\rm LAD}, L} - \hat{E}^p_{\psi_{\alpha}, L}}{\hat{E}^p_{ {\rm LAD}, L}}.
	\]
	$\Delta( {\rm LAD}, {\rm Huber})$ and $\Delta( {\rm Huber}, \psi_{\alpha})$ are defined similarly. The experiments are conducted under the following three different situations:	
	\begin{itemize}
		\item[(1)] For $\bfZ_t^1 \in \bbR^{20}$, we set $c_0 = 0.01$.  In the SGD \eqref{eq:SGD catoni}, we let $\alpha = 1.2$ and $\lambda = 0.020$, while in the SGD \eqref{eq:SGD huber}, we let $\tau = 0.5$ and $\sigma = 8$. The outcomes of these setups are illustrated in Figure \autoref{fig:Xt}.		
		\item[(2)] For $\bfZ_t^2 \in \bbR^{55}$, we set $c_0 = 0.03$. In the SGD \eqref{eq:SGD catoni}, we let $\alpha = 1.2$ and $\lambda = 0.040$, while in the SGD \eqref{eq:SGD huber}, we let $\tau = 0.2$ and $\sigma = 10$. The outcomes of these setups are illustrated in Figure \autoref{fig:Yt}.
		\item[(3)] For $\bfZ_t \in \bbR^{75}$, we set $c_0 = 0.04$. In the SGD \eqref{eq:SGD catoni}, we let $\alpha = 1.2$ and $\lambda = 0.080$, and in the SGD \eqref{eq:SGD huber}, we let $\tau = 0.15$ and $\sigma = 12$. The outcomes of these setups are illustrated in Figure \autoref{fig:Zt}.
	\end{itemize}
	
	\begin{figure*}[htpb]
		\centering
		\subfigure[ Simulation values for $\bfZ_t^1 \in \bbR^{20}$. \label{fig:Xt} ]{
			\begin{minipage}[t]{0.8\linewidth}
				\centering
				\includegraphics[width=\linewidth]{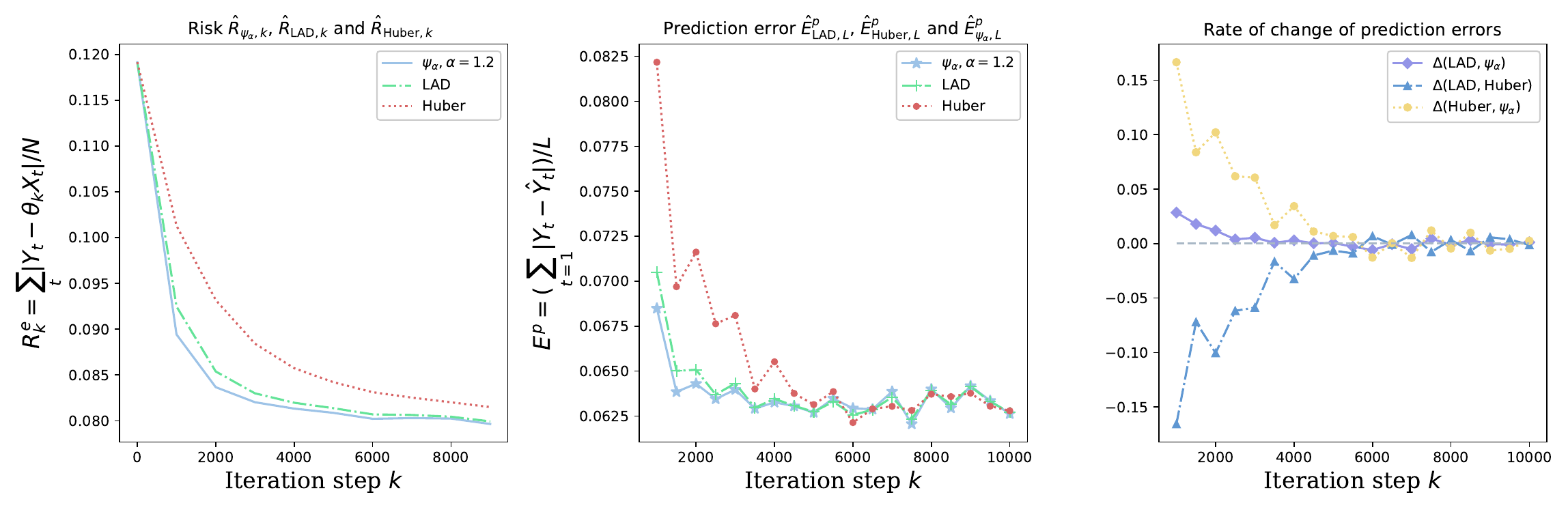}
			\end{minipage}
		}%
		
		\subfigure[ Simulation values for $\bfZ_t^2 \in \bbR^{55}$. \label{fig:Yt}]{
			\begin{minipage}[t]{0.8\linewidth}
				\centering
				\includegraphics[width=\linewidth]{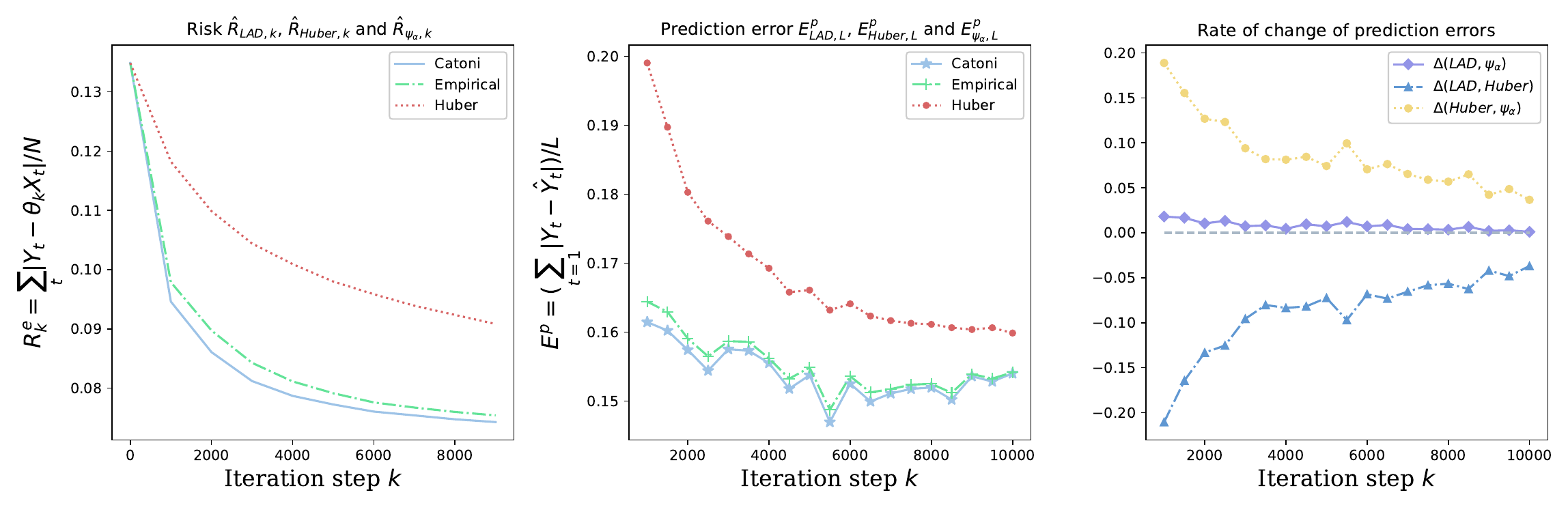}
			\end{minipage}
		}%
		
		\subfigure[ Simulation values for $\bfZ_t \in \bbR^{75}$. \label{fig:Zt} ]{
			\begin{minipage}[t]{0.8\linewidth}
				\centering
				\includegraphics[width=\linewidth]{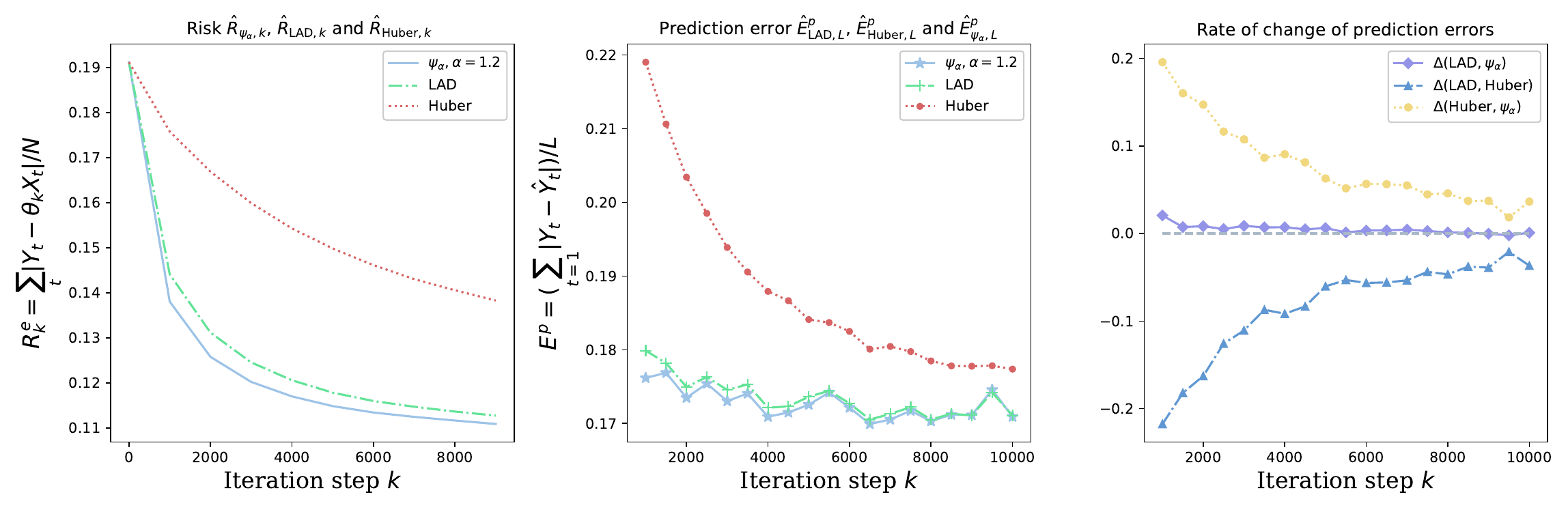}
			\end{minipage}
		}%
		
		\centering
		\caption{  For $\bfZ_t^1 \in \bbR^{20}$, $\bfZ_t^2 \in \bbR^{55}$ and  $\bfZ_t \in \bbR^{75}$, $t = 1, \dotsc, N = 190$, simulations values of risk $\hat{R}_{\psi_{\alpha},k}$, $\hat{R}_{LAD,k}$ and $\hat{R}_{Huber,k}$ with  $1 \le  k  \le  n = 10000$, the  prediction errors $\hat{E}^p_{\psi_{\alpha},L}$, $\hat{E}^p_{LAD,L}$ and $\hat{E}^p_{Huber,L}$ with $L=10$, and the rate of changes of prediction errors $\Delta({\rm LAD}, \psi_{\alpha})$, $\Delta({\rm LAD}, {\rm Huber})$ and $\Delta({\rm Huber}, \psi_{\alpha})$.  }
		\label{figure real}
	\end{figure*}	
	
	Based on \autoref{figure real}, we find that the minimization problem with  $\psi_{\alpha}$ truncation outperforms both the non-truncated approach and the Huber loss. Although the tail is not very heavy, the advantages of using  $\psi_{\alpha}$ are still evident. The Huber truncations do not perform well, we think it may be due to the small size of our data. We can see that when the dimension is high, the improvements of $\psi_\alpha$ truncation are significant.	
	
	\section*{Appendix}
	
	\appendix
	
	\section{Proof of \eqref{e:LADBigPro}}
	\label{s:LADBigPro}
	
	Recalling the definition of $R_{\ell_1}(\theta)$ in \eqref{e:R_1}, it is easy to verify that $R_{\ell_1}(\theta) - R_{\ell_1}(0)  \ge  \rho_n$ if and only if $\abs{\theta}  \ge  1$. Hence, it follows from $\rho_n = 1/\sqrt{n}$ that  
	\begin{equation*}
		\begin{aligned}
			\bbP\left( R_{\ell_1}(\bar{\theta}) - R_{\ell_1}(\theta^*)  \ge  \frac{1}{\sqrt{n}} \right) 
			= 2 \bbP \left( \bar{\theta}  \ge   1 \right).
		\end{aligned}
	\end{equation*}
	We claim that there exists a constant $n_0\in \mathbb{N}$ such that for all $n > n_0$,
	\begin{equation}
		\label{e:theta>1}
		\bbP \left( \bar{\theta}  \ge   1 \right)  \ge  \frac{9}{32 \rme^2},
	\end{equation}
	which implies \eqref{e:LADBigPro} immediately.
	
	Next, we verify \eqref{e:theta>1}. Recall that the population has the distribution \eqref{e:s_noise}. Then, the corresponding cumulative density function, denoted by $F_{Y}(x)$, is given by
	\begin{equation}\label{e:cdf_Y}
		F_{Y}(x) = \left\{
		\begin{aligned}
			\frac{1-\rho_n}{2 \abs{x}^{\mu}} , \quad & \text{as}\ x < -1, \\
			\frac{1-\rho_n}{2} , \quad & \text{as}\  -1  \le  x < 0 ,\\
			\frac{1+\rho_n}{2} , \quad & \text{as}\  0  \le  x < 1 ,\\
			1- \frac{1-\rho_n}{2 \abs{x}^{\mu}} , \quad & \text{as}\  x  \ge  1 .
		\end{aligned}
		\right.
	\end{equation} 
	Let $Y_{(r)}$ denote the $r$-th order statistic of the observations $\{ Y_i\}_{i=1}^{ 2m+1}$. According to \cite{MR3025012,MR1994955}, the cumulative density function of $Y_{(r)}$ is given by
	\[
	\bbP( Y_{(r)}  \le  x ) = \sum_{j=r}^{2m+1} 
	\begin{pmatrix}
		2m+1 \\ j
	\end{pmatrix}
	\left[ F_{Y}(x) \right]^j \left[ 1-F_{Y}(x) \right]^{2m+1 - j} .
	\]
	Since $\bar{\theta} = Y_{(m+1)}$ is the median, combining this with $F_Y(1) = (1+\rho_n)/2$ by \eqref{e:cdf_Y}, we have that
	\begin{equation}\label{e:pdf}
		\bbP( \bar{\theta}  \le  1 ) = \sum_{j=m+1}^{2m+1} 
		\begin{pmatrix}
			2m+1 \\ j
		\end{pmatrix}
		\left[ \frac{1+\rho_n}{2} \right]^j \left[ \frac{1-\rho_n}{2} \right]^{2m+1 - j}.
	\end{equation}
	
	Let $X\sim \mathrm{Bin}(2m+1, (1+\rho_n)/2)$ represent the binomial distribution with parameter $2m+1$ and $(1+\rho_n)/2$. We then observe that $\bbP( \bar{\theta}  \le  1 ) = \bbP(X \ge  m+1)$, which leads that
	\begin{equation}\label{e:aimP}
		\begin{aligned}
			\bbP \left( \bar{\theta}  \ge   1 \right) & \ge  1- \bbP( \bar{\theta}  \le  1 ) = \bbP(X  \le  m).
		\end{aligned}
	\end{equation} 
	On the other hand, let $\{Z_i\}_{i=1}^{ 2m+1} $ be a sequence of i.i.d. random variables following a Bernoulli distribution with parameter $(1+\rho_n)/2$. Then, we have that  
	\[
	\begin{aligned}
		\bbP(X  \le  m) &= \bbP\left( \frac{ \sum_{j=1}^{2m+1} Z_j - (2m+1) \bbE Z_1 }{ \left[ (2m+1) \mathrm{Var}(Z_1) \right] }  \le  - \frac{(2m+1)\rho_n + 1 }{\left[ (2m+1)(1-\rho_n^2) \right]^{1/2}} \right).
	\end{aligned}
	\]
	According to $\rho_n = 1/\sqrt{n}$ independent of $n$, it holds that
	\[
	1 < \frac{(2m+1)\rho_n + 1 }{\left[ (2m+1)(1-\rho_n^2) \right]^{1/2}} = \frac{\sqrt{2m+1}+1}{\sqrt{2m}} < 2.
	\]
	By Cram\'er type moderate deviation \cite{MR137142,MR388499,MR4580924} (the left side tail probability), as $m \ge m_0$ with $m_0 \in \mathbb N$ being some constant, we have
	\[
	\begin{aligned}
		&\pheq
		\left|\frac{\bbP(X  \le  m)}{\Phi\left( -(\sqrt{2m+1}+1)/\sqrt{2m} \right)}-1\right|  \le  \frac 14 ,
	\end{aligned}
	\]
	where $\Phi(x)$ is the cumulative density function of standard normal distribution. Hence, 
	\begin{equation}\label{e:aimP2}
		\begin{aligned}
			\bbP(X  \le  m)  \ge  \frac 34 \Phi\left(-\frac{ \sqrt{2m+1} + 1 }{\sqrt{2m}} \right) 
			\ge  \frac34 \Phi(-2) 
			\ge  \frac{9}{32 \rme^2},
		\end{aligned}
	\end{equation}
	where the last inequality is by $\Phi(-x) = 1-\Phi(x)  \ge  (x^{-1} - x^{-3}) \rme^{-x^2/2}$ for all $x>0$. 
	Combining these with \eqref{e:aimP} and \eqref{e:aimP2} leads to that as $n  \ge  n_0$ (recall $n=2m+1$) with some $n_0 \in \mathbb{N}$ being some constant, \eqref{e:theta>1} holds.
	
	\section{Auxiliary Lemmas}
	\label{sec:sec3}
	
	\begin{lemma}
		{\cite[Lemma 2.1]{MR0871254}}
		\label{lemma:beta-mixing}
		Let $\bfZ_1, \dotsc, \bfZ_n$ be real random variables (or vectors) on the same probability space. Define for $1 \le  i <n$
		\[
		\beta^{(i)} = \beta( \bfZ_i , ( \bfZ_{i+1} , \dotsc, \bfZ_n ) ).
		\]
		The probability space can be extended with random variables (or vectors respectively) $\widetilde{\bfZ}_i$ distributed as $\bfZ_i$ such that $\widetilde{\bfZ}_1, \dotsc, \widetilde{\bfZ}_n$ are independent and
		\[
		\bbP( \widetilde{\bfZ}_j \neq  \bfZ_j ,\ \text{for some}\ j=1,\dotsc,n )  \le  \beta^{(1)} + \cdots + \beta^{(n-1)}.
		\]
	\end{lemma}
	
	\begin{lemma}
		{\cite[Lemma 7]{ZhangAnru8368145}}\label{lem:cov_num of Matrix}
		Suppose $\norm{\cdot}$ is any matrix norm. Let $\bm{\Theta} = \{ \bmtheta \in \bbR^{d_2\times d_1} : \rank(\bmtheta)  \le  \kappa,  \norm{\bmtheta}  \le  1  \}$ be the class of low-rank matrices under norm $\norm{\cdot}$. Then there exists a $\delta$-net $\bm{\mathcal{N}}_{d_1,d_2}(\kappa)$ for $\bm{\Theta}$ with cardinality at most $\big((4+\delta)/\delta\big)^{(d_1+d_2)\kappa}$.
	\end{lemma}
	
	\section{$\beta$-mixing Conditions for VAR($p$) Model}
	\label{appendix: mixing}
	
	In order to apply the result in Section \ref{sec:sec2} to VAR($p$) model $\{ \bfZ_t, t \in \mathbb{Z} \}$ in \eqref{e:VAR(p)}, we need to verify it is an exponential $\beta$-mixing sequence. To this end,  
	let us first introduce the result in Pham and Tran \cite{PHAM1985297} briefly as below. 
	
	Assume that there exist a sequence of independent random vectors $\{ \bmepsilon_t, t\in \mathbb{Z}\}$ in $\bbR^d$ and matrices $\{\bfA_j \in \bbR^{d\times d}, j \in \mathbb{N} \}$ such that series $(\bfZ_t)_{t \ge  0}$ can be represented by
	\begin{equation}\label{def: Xt}
		\bfZ_t = \sum_{j=0}^{\infty} \bfA_j \bmepsilon_{t-j} , \quad \bfA_0 = \bfI_d .
	\end{equation}
	Assume that noise $\bmepsilon_t$ admits a density $g_t$.  Define $\gamma_j$ as 
	\begin{equation}
		\label{e:gamma(j)}
		\gamma_j = \sum_{k  \ge  j} \rho({\bfA_k}), \quad j  \ge  0 .
	\end{equation}
	Recall that $\rho(\bfA_k)$ denotes the maximum modulus of the eigenvalues of  matrix $\bfA_k$. 
	Let the following assumption hold, which can be seen in \cite[Lemma 2.2]{PHAM1985297}.
	
	\begin{assumption}
		\label{assump:coe.}
		Suppose that $g_t$ and $\bfA_j$ satisfy:
		\begin{itemize}
			\item[$(\runum{1})$] There exists some constant $M>0$ such that $ \int \abs{ g_t(\bfx-\bfy) - g_t(\bfx) } \dd \bfx < M \abs{\bfy} $ for all $t$;
			\item[$(\runum{2})$] For all $z \in \mathbb{C}$ with $\abs{z}  \le  1$,
			\[
			\sum_{j=0}^{\infty} \rho(\bfA_j) < \infty , \quad \text{and} , \quad \sum_{j=0}^{\infty} \bfA_j z^j  \neq \zero ,
			\]
		\end{itemize}
	\end{assumption}
	
	\begin{theorem}{\cite[Theorem 2.1]{PHAM1985297}}
		\label{thm:mixing condition}
		Let process $\bfZ_t$ be defined in \eqref{def: Xt}. Let Assumption \ref{assump:coe.} hold, and assume that there is a constant $K>0$ such that $\bbE \abs{ \bmepsilon_t}^\alpha < K$ for some $\alpha >0$ and for all $t$. If 
		\[
		\sum_{j=1}^{\infty} \gamma_j^{{\alpha} / {(1+\alpha)}} < \infty ,
		\]
		where $\gamma_j$ are defined as in \eqref{e:gamma(j)}. Then the $\beta$-mixing coefficient of $\bfZ^{ \le  0} = \big(\bfZ_k, k  \le  0\big)$ and $\bfZ^{ \ge  n} = \big( \bfZ_k, k  \ge  n \big)$ satisfies
		\begin{equation}
			\label{ab. regular}
			\beta(\bfZ^{ \le  0}, \bfZ^{ \ge  n})  \le  K \sum_{j=n}^{\infty}  \gamma_j^{{\alpha} / {(1+\alpha)}} .
		\end{equation}
	\end{theorem}
	
	Now let us come back to our VAR$(p)$ model in \eqref{e:VAR(p)}: 
	\begin{equation*}
		\bfZ_{t+1} = \bmPhi_1 \bfZ_t + \bmPhi_2 \bfZ_{t-1} + \cdots + \bmPhi_{p} \bfZ_{t+1-p} + \bm{\varepsilon}_{t+1},
	\end{equation*}
	where $\bmPhi_i \in \bbR^{d\times d}$, $i=1, \dotsc, p$, and $\bmepsilon_{t+1}$ is a sequence i.i.d. noises. Observe that 
	\[
	\begin{aligned}
		\bfY_{t+1} &:= 
		\begin{bmatrix}
			\bfZ_{t+1} \\
			\bfZ_t \\
			\vdots \\
			\bfZ_{t+2-p}
		\end{bmatrix}
		\! = \!
		\begin{bmatrix}
			\bmPhi_1   & \cdots  & \bmPhi_{p-1} & \bmPhi_{p} \\
			\bfI_d     & \cdots  &    \zero     & \zero \\
			\vdots     & \ddots  &    \vdots    & \vdots \\
			\zero      & \cdots  &    \bfI_d    & \zero
		\end{bmatrix} \!\!
		\begin{bmatrix}
			\bfZ_t \\
			\bfZ_{t-1} \\
			\vdots \\
			\bfZ_{t+1-p}
		\end{bmatrix}
		\!+\!
		\begin{bmatrix}
			\bmepsilon_{t+1} \\
			\zero \\
			\vdots \\
			\zero
		\end{bmatrix} 
		=: \bmPsi \cdot \bfY_t + \bmzeta_{t+1} ,
	\end{aligned}
	\]
	where the matrix $\bmPsi$ satisfies Assumption \ref{A6:stationary}. Then, a simple calculation yields that   
	\[
	\bfY_{t+1} = \sum_{j=0}^{\infty} \bmPsi^j \bmzeta_{t+1-j}, \quad \bmPsi^0 = \bfI_{pd}, 
	\]
	which leads to 
	\begin{equation*}
		\bfZ_{t+1} = \sum_{j=0}^{\infty} \bmPsi_j \bmepsilon_{t+1-j} , 
	\end{equation*}
	where $\bmPsi_0 = \bfI_d$, $\bmPsi_j = \zero$ for all $j <0$ , and  
	\begin{align*}
		\bmPsi_1 &= \bmPhi_1 , \\
		\bmPsi_2 &= \bmPhi_1 \bmPsi_1 + \bmPhi_2 , \\
		&\ \ \vdots \\
		\bmPsi_p &= \bmPhi_1 \bmPsi_{p-1} + \bmPhi_2 \bmPsi_{p-2} + \cdots + \bmPhi_{p} , \\
		\bmPsi_k &= \bmPhi_1 \bmPsi_{k-1} + \bmPhi_2 \bmPsi_{k-2} + \cdots + \bmPhi_{p} \bmPsi_{k-p} , \ k  \ge  p .
	\end{align*}
	
	Recall Assumption \ref{A6:stationary} and the matrix $\bm{\Psi}$ therein, for all $k\in \mathbb{N}$, there exists a constant $C_{\rm op}$ such that $\opnorm{\bmPsi^k}  \le  C_{\rm op} \rho^k$. Due to $\bmPsi_k$ is the submatrix of $ \bmPsi^k $ formed from rows $\{1,\dotsc,n\}$ and columns $\{1,\dotsc,n\}$, we know that
	\[
	\rho(\bmPsi_k)  \le  \opnorm{\bmPsi_k}  \le  \opnorm{\bmPsi^k}  \le  C_{\rm op} \rho^k, \quad \forall \ k \ge  0.
	\]
	Then $\gamma_j$ defined in \eqref{e:gamma(j)} satisfies
	\[
	\gamma_j =  \sum_{k  \ge  j}\rho( {\bmPsi_k} )  \le   C_{\rm op} \sum_{k  \ge  j} \rho^k  \le  C_{\rm op} \frac{\rho^j}{1-\rho},
	\]
	which leads to $\sum_{j=1}^{\infty } \gamma_j^{{\alpha}/{(1+\alpha)}} < \infty$ immediately. Then, if Assumption \ref{A5:noise} holds additionally, we can apply Theorem \ref{thm:mixing condition} to obtain that
	\[
	\beta(n)  \le  K \sum_{j=n}^{\infty} \gamma_j^{{\alpha}/{(1+\alpha)}}  \le  \frac{K C_{\rm op}^{\alpha/(1+\alpha)}}{ (1-\rho)^{\alpha/(1+\alpha)} (1-\rho^{\alpha/(1+\alpha)}) } \cdot \exp\left\{ -n\cdot \frac{\alpha}{1+\alpha} \log \frac{1}{\rho}\right\} ,
	\]
	for all $\alpha \in (1,2]$. $\alpha/(1+\alpha) > 1/2$ and $\rho \in(0,1)$ yield Lemma \ref{lem:mixing VARp} immediately.		
	
	\section{Proofs of Theorem \ref{thm:VAR} }
	\label{sec: proof}
	
	\begin{proof}
		[\textbf{Proof of Theorem \ref{thm:VAR}}]
		This is a direct application of Theorem \ref{cor:M_n = m_n = log n}. We need to estimate the corresponding covering number.
		Recall the definition $\bm{\Theta}$ defined in \eqref{e:Xi}, $i=1, \dotsc, p$, we then know that
		\begin{gather*}
			\rank([\bmPhi_1, \dotsc, \bmPhi_p] )  \le  d,  \\
			\big\|{ [\bmPhi_1, \dotsc, \bmPhi_p] }\big\|_{\rm op}  \le  \opnorm{ \bmPsi}  \le  C_{\rm op} \rho , 
		\end{gather*}
		which leads to
		\[
		\sum_{i = 1}^p \norm{\bmPhi_i}_{1,1} = \norm{ [\bmPhi_1, \dotsc, \bmPhi_{p}] }_{1,1}  \le  C_{\rm op} \sqrt{pd(d+1)} \rho.
		\]
		Additionally, it follows from \cite[Lemma 7]{ZhangAnru8368145} that for all $\delta \in (0,1)$
		\[
		N\big( \bm{\Theta} , \delta \big)  \le  \left( \frac{ 6 C_{\rm op} \rho }{\delta} \right)^{p(p+1)d^2} .
		\]
		Then,applying Theorem \ref{cor:M_n = m_n = log n} implies that the following inequality
		\[
		\begin{aligned}
			\popR(\widehat{\bm{\theta}}) - \popR(\bm{\theta}^*) 
			& \le  C \left( \frac{\log n}{ \abs{\log \rho} n} \big( \abs{\log \varepsilon} + p^2 d^2  \log n \big) \right)^{(\alpha-1)/\alpha} ,
		\end{aligned}
		\]
		holds with probability at least $1-2\varepsilon$ for some constant $C>0$ independent of $\varepsilon, n, p, \rho$ and $\alpha$. 
		We complete the proof.
	\end{proof}

	\section{Proofs of Auxiliary Lemmas}
	\label{appendix}
	
	To make notations simple, we define functions $\ell_1$ and $\ell_{\alpha}$ as
	\begin{equation*}
		\label{def:l_1,l_alpha}
		\ell_1(\bfy,\bfx,\bmtheta) = \abs{\bfy - \bmtheta \cdot \bfx }, 
		\quad \text{and } \quad 
		\ell_\alpha(\bfy,\bfx,\bmtheta) = \abs{ \bfy - \bmtheta \cdot \bfx }^\alpha .
	\end{equation*} 
	Besides, recalling that the $\ell_1$-risk and $\ell_\alpha$-risk are given by
	\begin{equation*}
		\label{def:risk}
		R_{\ell_1}(\bmtheta) = \bbE_{(\bfX, \bfY) \sim \bm{\Pi}}\big[ \ell_1( \bfY,\bfX,\bmtheta) \big]
		\quad \text{and} \quad
		R_{\ell_\alpha}(\bmtheta) = \bbE_{(\bfX, \bfY) \sim \bm{\Pi}}\big[ \ell_\alpha( \bfY,\bfX,\bmtheta) \big].
	\end{equation*} 
	And the truncation function $\psi_{\alpha}(r): \bbR \to \bbR$ is non-decreasing and satisfies 
	\[
	-\log\left( 1 - r + \frac{\abs{r}^{\alpha}}{\alpha} \right)  \le  \psi_\alpha(r)  \le  \log\left( 1 + r + \frac{\abs{r}^{\alpha}}{\alpha} \right), \quad \alpha \in (1,2]. 
	\]
	
	\begin{proof}
		[\textbf{Proof of Lemma \ref{lem:C_Ine}}]
		Let $ \mathcal{I} \subseteq \{1,\dotsc, n\}$ be any subset. For \eqref{e:upper}, $\psi_\alpha(r)  \le  \log(1+r+\abs{r}^\alpha / \alpha) $ leads to
		\begin{align*}
			&\pheq
			\bbE\left[ \exp\bigg\{ \frac{1}{\abs{\mathcal{I}}} \sum_{i \in \mathcal{I}}  \psi_{\alpha}  \big( \lambda \abs{ \bfY_i - \bmtheta \cdot \bfX_i  } \big) \bigg\} \right] \\
			& \le  \bbE\left[ \prod_{i \in \mathcal{I}} \big( 1 + \lambda \ell_1(\bfY_i, \bfX_i, \bmtheta) + \alpha^{-1} \lambda^{\alpha} \ell_{\alpha}(\bfY_i, \bfX_i, \bmtheta) \big)^{1/\abs{\mathcal{I}}} \right] \\
			& \le  \big( 1 + \lambda R_{\ell_1}(\bmtheta) + \alpha \lambda^{\alpha} R_{\ell_{\alpha}}(\bmtheta) \big) 
			\le  \exp\big\{ \lambda R_{\ell_1}(\bmtheta) + \alpha \lambda^{\alpha} R_{\ell_{\alpha}}(\bmtheta) \big\},
		\end{align*}
		where the second inequality is by the general H\"older's inequality. 
		Analogously, by $\psi_{\alpha}(r)  \ge  - \log ( 1 - r + \abs{r}^{\alpha} / \alpha)$, and $(a+b)^{\alpha}  \le  2^{\alpha -1}(a^\alpha + b^{\alpha})$ for any $a,b  \ge  0$ with $\alpha\in(1,2]$, we have
		\begin{align*}
			&\pheq
			\bbE\left[ \exp\left\{ -\frac{1}{\abs{\mathcal{I}}} \sum_{i \in \mathcal{I}}  \psi_{\alpha} \left( \lambda \abs{\bfY_i - \bmtheta \cdot \bfX_i }  - \lambda \delta \abs{\bfX_i} \right) \right\} \right] \\
			& \le  \bbE \left[ \prod_{i \in \mathcal{I}} \Big( 1 - \lambda \ell_1(\bfY_i,\bfX_i,\bmtheta) + \lambda \delta \abs{\bfX_i} +  \frac{\lambda^{\alpha}}{\alpha} \big|{ \ell_1(\bfY_i,\bfX_i,\bmtheta) - \delta \abs{\bfX_i} }\big|^{\alpha} \Big)^{ 1/{\abs{\mathcal{I}}} } \right] \\
			& \le  1 + \lambda\left[ - \popR(\bmtheta) + \delta \bbE\abs{\bfX_1} + \frac{(2\lambda)^{\alpha-1}}{\alpha} \Big( \sup_{\bmtheta\in\bm{\Theta}} R_{\ell_{\alpha}}(\bmtheta) + \delta^{\alpha} \bbE\abs{\bfX_1}^{\alpha} \Big) \right],
		\end{align*}
		which yields \eqref{e:lower} immediately.
	\end{proof}
	
	\begin{proof}[\textbf{Proof of Lemma \ref{lemma:error bound w.r.t. theta star}}]		
		To make notations simple, we write function $g$ as
		\begin{equation}
			\label{def:g}
			g(n, \alpha,\lambda,\bmtheta^*,\varepsilon) = \popR(\bmtheta^*) + \frac{\lambda^{\alpha-1}}{\alpha} R_{\ell_{\alpha}}(\bmtheta^*) + \frac{2 M_n }{\lambda K(n)m_n} \log\frac{4}{\varepsilon}
		\end{equation}
		and let
		\[
		\catoniR(\bmtheta^*) 
		=
		\frac{1}{\lambda n} \sum_{i=1}^n  \psi_{\alpha} \left( \lambda \abs{\bfY_i -  \bmtheta^* \cdot \bfX_i } \right) =: \frac1{\lambda n} S(\bmtheta^*).
		\]
		$S(\bmtheta^*)$ can be divided into three parts, the big blocks sum $S_b(\bmtheta^*)$, the small blocks sum $S_s(\bmtheta^*)$ and the remainder sum $S_r(\bmtheta^*)$, that is
		\[
		S(\bmtheta^*) =  S_b(\bmtheta^*) + S_s(\bmtheta^*) + S_r(\bmtheta^*) = \sum_{j=1}^{K(n)} S_{j,b}(\bmtheta^*) + \sum_{j=1}^{K(n)} S_{j,s}(\bmtheta^*) + S_r(\bmtheta^*) ,
		\]
		where $S_{j,b}(\bmtheta^*)$, $S_{j,s}(\bmtheta^*)$ and $ S_r(\bmtheta^*)$ are given by
		\begin{align*}
			S_{j,b}(\bmtheta^*) &=  \sum_{i \in \mathcal{J}_{j,M_n}}  \psi_{\alpha} \big( \lambda \abs{\bfY_i -  \bmtheta^* \cdot \bfX_i  } \big) , \\
			S_{j,s}(\bmtheta^*) &=  \sum_{i \in \mathcal{I}_{j,m_n}}  \psi_{\alpha} \big( \lambda \abs{\bfY_i -  \bmtheta^* \cdot \bfX_i } \big) , \\
			S_{r}(\bmtheta^*) &= \sum_{i \in \mathcal{R}_n}  \psi_{\alpha} \big( \lambda \abs{\bfY_i -  \bmtheta^* \cdot \bfX_i  } \big) . 
		\end{align*}
		Here, $\mathcal{J}_{j,M_n}$ and $\mathcal{I}_{j,m_n}$, $j = 1, \dotsc, K(n)$, are interlacing blocks given in \eqref{def:blocks}, and $\mathcal{R}_n$ is given in \eqref{def:reserved block}.
		On the other hand, we have the following inequality 
		\begin{align*}
			&\pheq
			\bbP\left( \catoniR(\bmtheta^*)  \ge  g(n, \alpha,\lambda,\bmtheta^*,\varepsilon) \right) \\
			&= \bbP\left( \sum_{i=1}^n  \psi_{\alpha} \big( \lambda \abs{ \bfY_i- \bmtheta^* \cdot \bfX_i } \big)  \ge  \lambda \big( K(n)M_n + K(n)m_n + R_n \big) g(n, \alpha,\lambda,\bmtheta^*,\varepsilon) \right) \\
			& \le  \bbP\big( S_b(\bmtheta^*) + S_r(\bmtheta^*)  \ge  \lambda ( K(n)M_n + R_n) \cdot g(n, \alpha,\lambda,\bmtheta^*,\varepsilon) \big) \\
			&\pheq + \bbP\big( S_s(\bmtheta^*)  \ge  \lambda K(n)m_n \cdot g(n, \alpha,\lambda,\bmtheta^*,\varepsilon) \big)  =: {\rm I}_1 + {\rm I}_2.
		\end{align*}
		We claim that for any $\varepsilon\in(0,1)$, the followings hold
		\begin{align}
			{\rm I}_1 &:= \bbP\big( S_b(\bmtheta^*) + S_r(\bmtheta^*)  \ge  \lambda ( K(n)M_n + R_n) \cdot g(n, \alpha,\lambda,\bmtheta^*,\varepsilon) \big)  \le  \frac{\varepsilon}{2} \label{ineq i } \\
			{\rm I}_2 &:= \bbP\big( S_s(\bmtheta^*)  \ge  \lambda K(n)m_n \cdot g(n, \alpha,\lambda,\bmtheta^*,\varepsilon) \big)  \le  \frac{\varepsilon}{2} . \label{ineq:ii}
		\end{align}
		Combining these with definition of function $g$ in \eqref{def:g} implies the desired result. 
		
		It remains to prove \eqref{ineq i } and \eqref{ineq:ii}. We only verify \eqref{ineq i } in detail, since \eqref{ineq:ii} can be done in the same way.
		Recall the sequence of random vectors $\{ \bm{\mathcal{H}}_j \}_{1\leq j \leq K(n)}$ and $\bm{\mathcal{R}}$ as defined in \eqref{def:HGR}, and the corresponding independent sequence $\{ \widetilde{\bm{\mathcal{H}}}_j \}_{1\leq j \leq K(n)}$ and $\widetilde{\bm{\mathcal{R}}}$ as defined in \eqref{def:ind HGR}. Additionally, recall the event $\mathcal{A}$ as defined in \eqref{def: events A & B},
		\[
		\mathcal{A} = \{ \widetilde{\bm{\mathcal{H}}}_{j} \neq \bm{\mathcal{H}}_{j} \text{ for some } 1 \le  j  \le  K_n , \ \text{or} \ \widetilde{\bm{\mathcal{R}}} \neq \bm{ \mathcal{R}} \} .    
		\]
		Under the conditions in the lemma, it satisfies that
		\[
		\bbP(\mathcal{A} )  \le  BK(n)\rme^{-\beta m_n}  \le  \frac{\varepsilon}{4}.
		\]
		We then define $\widetilde{S}_{j,b}(\bmtheta^*)$, $j = 1, \dotsc, K(n) $, and $\widetilde{S}_r(\bmtheta^*)$ as following:
		\[
		\widetilde{S}_{j,b}(\bmtheta^*) = \sum_{i \in \mathcal{J}_{j,M_n}}  \psi_{\alpha} \big( \lambda |\widetilde{\bfY}_i -  \bmtheta^* \cdot \widetilde{\bfX}_i   | \big), 
		\quad
		\widetilde{S}_r(\bmtheta^*) = \sum_{i \in \mathcal{R}_{R_n}}  \psi_{\alpha} \big( \lambda |\widetilde{\bfY}_i -  \bmtheta^* \cdot \widetilde{\bfX}_i   | \big) ,
		\]
		which are independent random variables have the same distributions as $S_{j,b}(\bmtheta^*)$, $j = 1, \dotsc, K(n) $, and $S_r(\bmtheta^*)$ respectively. Additionally, It holds that  
		\[
		\bbP\big( \widetilde{S}_{j,b} \neq S_{j,b} \ \text{ for some}  \ 1 \le  j  \le  K(n) , \ \text{or}  \ \widetilde{S}_r \neq S_r \big)  \le  \bbP(\mathcal{A})  \le  \frac{\varepsilon}{4} .
		\]
		Besides, we also have 
		\begin{align}\label{ineq:lemma 1.2}
			&\pheq
			\bbE\left[ \exp\Bigg\{ \frac{ 1 }{2 M_n }  \bigg( \sum_{j=1}^{K(n)} \widetilde{S}_{j,b}(\bmtheta^*) +  \widetilde{S}_{r}(\bmtheta^*) \bigg)\Bigg\} \right] \nonumber \\
			&= 
			\prod_{j=1}^{K(n)} \bbE\left[ \exp\left\{ \frac{1}{2 M_n } \widetilde{S}_{j,b}(\bmtheta^*) \right\} \right] 
			\cdot \bbE\left[ \exp\left\{ \frac{1}{2 M_n } \widetilde{S}_{r}(\bmtheta^*) \right\} \right] \\
			& \le  \prod_{j=1}^{K(n)}  \left\{ \bbE\left[ \exp\left\{ \frac{1}{M_n} S_{j,b}(\bmtheta^*) \right\} \right] \right\}^{1/2} 
			\left\{ \bbE\left[ \exp\left\{ \frac{1}{R_n} S_{r}(\bmtheta^*) \right\} \right] \right\}^{{R_n}/{ (2 M_n) }} \nonumber \\ 
			& \le  \exp\left\{ \frac{K(n)M_n + R_n }{2 M_n } \big( \lambda R_{\ell_1}(\bmtheta^*) + \alpha \lambda^{\alpha} R_{\ell_{\alpha}}(\bmtheta^*) \big) \right\}. \nonumber 
		\end{align}
		where the first inequality is by Jensen's inequality and the last inequality comes from \eqref{e:upper}. We define events $\mathcal{C}$ and $\mathcal{D}$ as the following:
		\begin{equation*}
			\begin{aligned}
				\mathcal{C} &= \big\{   S_b(\bmtheta^*) + S_r(\bmtheta^*)  \ge  \lambda ( K(n)M_n + R_n) \cdot g(n, \alpha,\lambda,\bmtheta^*,\varepsilon)  \big\} , \\
				\mathcal{D} &=  \big\{  \widetilde{S}_{j,b}(\bmtheta^*) \neq S_{j,b}(\bmtheta^*) \ \text{ for some}  \ 1 \le  j  \le  K(n) , \ \text{or}  \ \widetilde{S}_r(\bmtheta^*) \neq S_r(\bmtheta^*)  \big\} .
			\end{aligned}            
		\end{equation*}
		Then, we have
		\begin{equation}
			\label{e:I1}
			{\rm I}_1 = \bbP( \mathcal{C} \cap \mathcal{D} ) +  \bbP( \mathcal{C} \cap \mathcal{D}^c )  \le  \frac{\varepsilon}{4} + \bbP( \mathcal{C} \cap \mathcal{D}^c ).
		\end{equation}
		Thus, to prove \eqref{ineq i }, it suffices to verify $\bbP( \mathcal{C} \cap \mathcal{D}^c ) \le  \varepsilon/4$. By Markov's inequality, \eqref{def:g} and \eqref{ineq:lemma 1.2},
		\begin{align*}
			\bbP( \mathcal{C} \cap \mathcal{D}^c )		
			&=  \bbP\left( \sum_{j=1}^{K(n)} \widetilde{S}_{j,b}(\bmtheta^*) + \widetilde{S}_r(\bmtheta^*)  \ge  \lambda  K(n)M_n \cdot g(n, \alpha,\lambda,\bmtheta^*,\varepsilon) \right)  \\
			& \le  \frac{ \bbE\left[ \exp\left\{ \Big( \sum_{j=1}^{K(n)} \widetilde{S}_{j,b}(\bmtheta^*) + \widetilde{S}_r(\bmtheta^*) \Big)/{2 M_n } \right\} \right] }{\exp\left\{ \frac{\lambda ( K(n)M_n + R_n )}{2 M_n } g(n,\alpha,\lambda,\bmtheta^*,\varepsilon) \right\}} \\
			& \le  \exp\left\{ -\frac{K(n)M_n + R_n}{K(n)m_n} \log\frac{4}{\varepsilon} \right\}   \le  \frac{\varepsilon}{4}, 
		\end{align*}
		where the last inequality is due to the choice $m_n \le M_n$. The verification is \eqref{ineq i } completed. 
	\end{proof}
	
	Before proving Lemma \ref{lemma: error bound of theta hat}, we give the following lemma.
	
	\begin{lemma}
		\label{lemma: bound wrt delta net}
		Keep the same conditions in Lemma \ref{lemma: error bound of theta hat}. Then, the following inequality holds for any $\tilde{\bmtheta} \in \bm{\mathcal{N}}(\bm{\Theta},\delta)$
		\[
		-\frac1{n\lambda} \sum_{i=1}^n \psi_\alpha \left( \lambda \big|{\bfY_i - \tilde{\bmtheta} \cdot  \bfX_i }\big| - \lambda \delta \abs{\bfX_i} \right) 
		\le  -\popR(\tilde{\bmtheta}) + h(n,\alpha,\lambda,\delta,\varepsilon )
		\]
		with probability at least $1-\varepsilon$, where the function $h$ is given in \eqref{eq:h}, that is,
		\begin{equation*}
			\begin{aligned}
				h(n,\alpha,\lambda,\delta,\varepsilon)
				&= \delta \bbE\abs{\bfX_1} + \frac{(2\lambda)^{\alpha-1}}{\alpha}\left( \sup_{\bmtheta\in\bm{\Theta}} R_{\ell_{\alpha}}(\bmtheta) + \delta^{\alpha} \bbE\abs{\bfX_1}^{\alpha} \right) \\
				&\pheq +  \frac{2 M_n}{\lambda K(n)m_n} \log\frac{4 N(\bm{\Theta},\delta)}{\varepsilon} .
			\end{aligned}
		\end{equation*}
	\end{lemma}
	
	\begin{proof}[\textbf{Proof of Lemma \ref{lemma: bound wrt delta net}}] 
		The proof is very similar to the proof of Lemma \ref{lemma:error bound w.r.t. theta star}, but we sketch out necessary details here. To make notations simple, for any fixed $\tilde{\bmtheta} \in \bm{\mathcal{N}}(\bm{\Theta},\delta) = \{ \tilde{\bmtheta}_i, i = 1,\dotsc, N(\bm{\Theta}, \delta) \}$ being the $\delta$-net of $\bm{\Theta}$, denote 
		\[
		T(\tilde{\bmtheta}) = - \sum_{i=1}^n  \psi_{\alpha} \left( \lambda \big|{\bfY_i - \tilde{\bmtheta} \cdot  \bfX_i }\big| - \lambda \delta \abs{\bfX_i} \right) ,
		\]
		and denote the event $\mathcal{C}(\tilde{\bmtheta})$ as 
		\[
		\mathcal{C}(\tilde{\bmtheta}) = \left\{\frac1{n\lambda} T(\tilde{\bmtheta})  \ge  -\popR(\tilde{\bmtheta}) + h(n,\alpha,\lambda,\delta,\varepsilon) \right\}.
		\] 
		We claim that
		\begin{equation}
			\label{lem 3.3 claim}
			\bbP\big(\mathcal{C}(\tilde{\bmtheta}_i)\big)  \le  \frac{\varepsilon }{ N(\bm{\Theta},\delta) }, \quad \forall \ i = 1, \dotsc, N(\bm{\Theta},\delta),
		\end{equation} 
		which will implies the result immediately, that is
		\[
		\bbP\bigg( \bigcap_{i=1}^{N(\bm{\Theta},\delta)} \mathcal{C}(\tilde{\bmtheta}_i)^c \bigg)  \ge  1 - \sum_{i = 1}^{N(\bm{\Theta},\delta)} \bbP\big(\mathcal{C}(\tilde{\bmtheta}_i)\big)  \ge  1 - \varepsilon.
		\]
		
		Next, we prove \eqref{lem 3.3 claim} for any fixed $\tilde{\bmtheta} \in \bm{\mathcal{N}}(\bm{\Theta},\delta)$. Analogous to the proof of Lemma \ref{lemma:error bound w.r.t. theta star},  we divide $T(\tilde{\bmtheta})$ into three parts, that is
		\[
		T(\tilde{\bmtheta}) 
		= T_b(\tilde{\bmtheta}) + T_s(\tilde{\bmtheta}) + T_r(\tilde{\bmtheta}) 
		:= \sum_{j=1}^{K(n)}  T_{j,b}(\tilde{\bmtheta}) + \sum_{j=1}^{K(n)} T_{j,s}(\tilde{\bmtheta})  + T_r(\tilde{\bmtheta}),
		\]
		where 
		\begin{align*}
			T_{j,b}(\tilde{\bmtheta}) &= -\sum_{i \in \mathcal{J}_{j,M_n}}  \psi_{\alpha} \left( \lambda \big|{\bfY_i - \tilde{\bmtheta} \cdot  \bfX_i } \big| - \lambda \delta \abs{\bfX_i} \right) , \\
			T_{j,s}(\tilde{\bmtheta}) &= -\sum_{i \in \mathcal{I}_{j,m_n}}  \psi_{\alpha} \left( \lambda \big|{\bfY_i - \tilde{\bmtheta} \cdot  \bfX_i } \big| - \lambda \delta \abs{\bfX_i} \right) , \\
			T_{r}(\tilde{\bmtheta}) &= -\sum_{i \in \mathcal{R}_n}  \psi_{\alpha} \left( \lambda \big|{\bfY_i - \tilde{\bmtheta} \cdot  \bfX_i } \big| - \lambda \delta \abs{\bfX_i} \right) .
		\end{align*}
		Then, we have that 
		\[
		\begin{aligned}
			\bbP\big(\mathcal{C}(\tilde{\bmtheta})\big) &= \bbP \big( T(\tilde{\bmtheta})  \ge  \lambda(K(n)M_n+ K(n) m_n + R_n)[-\popR(\tilde{\bmtheta}) + h(n,\alpha,\lambda,\delta,\varepsilon)] \big) \\
			& \le  \bbP\big( T_b(\tilde{\bmtheta}) + T_r(\tilde{\bmtheta})  \ge  \lambda(K(n)M_n + R_n)[-\popR(\tilde{\bmtheta}) + h(n,\alpha,\lambda,\delta,\varepsilon)] \big) \\
			&\pheq \bbP\big( T_s(\tilde{\bmtheta})   \ge  \lambda K(n)m_n [-\popR(\tilde{\bmtheta}) + h(n,\alpha,\lambda,\delta,\varepsilon)] \big) =: {\rm J}_1 + {\rm J}_2.
		\end{aligned}
		\]
		We only prove that 
		\begin{equation*}
			{\rm J}_1 = \bbP\big( T_b(\tilde{\bmtheta}) + T_r(\tilde{\bmtheta})  \ge  \lambda(K(n)M_n + R_n)[-\popR(\tilde{\bmtheta}) + h(n,\alpha,\lambda,\delta,\varepsilon)] \big)  \le  \frac{\varepsilon }{ 2 N(\bm{\Theta},\delta) },
		\end{equation*}
		since one can prove ${\rm J}_2  \le  \varepsilon / (2  N(\bm{\Theta},\delta))$ in the same way. 		
		Let
		\[
		\begin{aligned}
			\widetilde{T}_{j,b}(\tilde{\bmtheta}) &= - \sum_{i \in \mathcal{J}_{j,M_n}}  \psi_{\alpha} \left( \lambda |\widetilde{\bfY}_i - \tilde{\bmtheta} \cdot  \widetilde{\bfX}_i  | - \lambda \delta |\widetilde{\bfX}_i|\right), \quad j=1,\dotsc,K(n),\\
			\widetilde{T}_r(\tilde{\bmtheta}) &= - \sum_{i \in \mathcal{R}_{R_n}}  \psi_{\alpha} \left( \lambda |\widetilde{\bfY}_i - \tilde{\bmtheta} \cdot  \widetilde{\bfX}_i  | - \lambda \delta |\widetilde{\bfX}_i|\right)
		\end{aligned}
		\]
		be the sum on $\widetilde{\bm{\mathcal{H}}}_j$ and $\bm{\widetilde{\mathcal{R}}}$ as defined in \eqref{def:ind HGR}, and they are independent random variables having same distributions as $T_{j,b}(\tilde{\bmtheta})$, $j = 1, \dotsc, K(n) $, and $T_r(\tilde{\bmtheta})$ respectively. Define events $\mathcal{D}(\tilde{\bmtheta})$ and $\mathcal{E}(\tilde{\bmtheta})$ as
		\[
		\begin{aligned}
			\mathcal{D}(\tilde{\bmtheta}) &= \{ \widetilde{T}_{j,b}(\tilde{\bmtheta}) \neq T_{j,b}(\tilde{\bmtheta}) \ \text{ for some}  \ 1 \le  j  \le  K(n) , \ \text{or}  \ \widetilde{T}_r(\tilde{\bmtheta}) \neq T_r(\tilde{\bmtheta}) \}  \\
			\mathcal{E}(\tilde{\bmtheta}) &= \{ T_b(\tilde{\bmtheta}) + T_r(\tilde{\bmtheta})  \ge  \lambda(K(n)M_n + R_n)[-\popR(\tilde{\bmtheta}) + h(n,\alpha,\lambda,\delta,\varepsilon)] \}
		\end{aligned}            
		\]
		Then, it holds that
		\[
		\bbP\big( \mathcal{D}(\tilde{\bmtheta}) \big)  \le  \bbP(\mathcal{A})  \le  BK(n)\rme^{-\beta m_n}  \le  \frac{\varepsilon}{4  N(\bm{\Theta},\delta)}.
		\]
		Additionally, by Markov's inequality, \eqref{e:lower} and the definition of function $h$, we have
		\begin{align*}
			{\rm J}_1 &= \bbP\big(  \mathcal{D}(\tilde{\bmtheta}) \cap \mathcal{E}(\tilde{\bmtheta}) \big) + \bbP\big(  \mathcal{D}(\tilde{\bmtheta})^c \cap \mathcal{E}(\tilde{\bmtheta}) \big) \\
			& \le  
			\frac{ \bbE\left[ \exp\left\{ \Big( \sum_{j=1}^{K(n)} {\widetilde{T}_{j,b}(\tilde{\bmtheta})} +  \widetilde{T}_r(\tilde{\bmtheta}) \Big) / {\big( 2 M_n \big) }  \right\} \right] }{  \exp\left\{ \frac{  \lambda ( K(n) M_n + R_n ) }{2 M_n } \big( -\popR(\tilde{\bmtheta}) + h(n,\alpha,\lambda,\delta,\varepsilon) \big) \right\} } + \frac{\varepsilon}{4  N(\bm{\Theta},\delta)} \\
			& \le  \exp\left\{ -\frac{K(n)M_n+R_n}{K(n)m_n} \log \frac{4N(\bm{\Theta},\delta)}{\varepsilon} \right\} + BK(n)\rme^{-\beta m_n} \nonumber \\
			& \le  \frac{\varepsilon}{4 N(\bm{\Theta},\delta) } + \frac{\varepsilon}{4  N(\bm{\Theta},\delta)}  \le  \frac{\varepsilon}{2 N(\bm{\Theta},\delta)}. \nonumber
		\end{align*}
		
		The proof is complete.
	\end{proof}
	
	Finally, we apply this lemma to prove Lemma \ref{lemma: error bound of theta hat}.
	
	\begin{proof}[\textbf{Proof of Lemma \ref{lemma: error bound of theta hat}}]
		Since $\hat{\bmtheta}\in\bm{\Theta}$, there exists  a $\tilde{\bmtheta} \in \bm{\mathcal{N}}(\bm{\Theta},\delta)$ such that
		\[
		\big\| {\hat{\bmtheta} - \tilde{\bmtheta}} \big\|_{\rm op}  \le  \delta.
		\] 
		Then, we have the following decomposition,
		\[
		\popR(\hat{\bmtheta}) - \catoniR(\hat{\bmtheta}) = \left[ \popR(\hat{\bmtheta}) - \popR(\tilde{\bmtheta}) \right] + \left[ \popR(\tilde{\bmtheta}) - \catoniR(\hat{\bmtheta}) \right].
		\]
		By the definition of $\popR(\bmtheta)$ and triangle inequality, we also have
		\begin{equation}\label{ineq: lem 3.4_1}
			\popR(\hat{\bmtheta}) - \popR(\tilde{\bmtheta}) 
			= \bbE_{(\bfX_1, \bfY_1)\sim \bm{\Pi}} \big[ |\bfY_1- \hat{\bmtheta} \cdot \bfX_1 | -  |\bfY_1- \tilde{\bmtheta} \cdot \bfX_1  | \big]  \le  \delta \bbE\abs{\bfX_1}.
		\end{equation}
		On the other hand,  for any $i=1,\dotsc,n$
		\[
		\big|{\bfY_i- \hat{\bmtheta} \cdot \bfX_i}\big|  \ge  \big|{\bfY_i- \tilde{\bmtheta} \cdot \bfX_i }\big| - \delta \abs{\bfX_i}.
		\]
		Hence, combining non-decreasing property of $ \psi_{\alpha} $, we have
		\begin{equation}\label{ineq: lem 3.4_2}
			\catoniR(\hat{\bmtheta}) 
			\ge  \frac1{n\lambda} \sum_{i=1}^n  \psi_{\alpha}  \left(\lambda \big|  \bfY_i-\tilde{\bmtheta} \cdot \bfX_i \big| - \lambda \delta \abs{\bfX_i}  \right) .
		\end{equation}
		Then,  applying Lemma \ref{lemma: bound wrt delta net} yields that
		\begin{align*}
			&\pheq
			\bbP\left( \popR(\hat{\bmtheta}) - \catoniR(\hat{\bmtheta}) 
			\le 
			\delta \bbE\abs{\bfX_1} + h(n,\alpha,\lambda, \delta, \varepsilon, \tilde{\bmtheta}) \right) \nonumber \\
			& \ge  \bbP\left( \catoniR(\hat{\bmtheta} )  \ge  \popR(\tilde{\bmtheta}) - h(n,\alpha,\lambda,\delta,\varepsilon) \right) \\
			& \ge  
			\bbP\left( \frac1{n\lambda} \sum_{i=1}^n  \psi_{\alpha}  \left(\lambda \big|\bfY_i-\tilde{\bmtheta} \cdot \bfX_i \big| - \lambda \delta \abs{\bfX_i}  \right)  \ge  \popR(\tilde{\bmtheta}) - h(n,\alpha,\lambda,\delta,\varepsilon) \right)  \nonumber \\
			& \ge  1-\varepsilon. \nonumber
		\end{align*}			
		where the first inequality is by \eqref{ineq: lem 3.4_1}, the second inequality is by \eqref{ineq: lem 3.4_2}. 
		The proof is completed.
	\end{proof}
	
	
	\section*{Acknowledgements}
	G. Li is supported by Hong Kong Research Grant Council grant 17313722. 
	
	L. Xu is supported by the National Natural Science Foundation of 
	China No. 12071499, The Science and Technology Development Fund (FDCT) of Macau S.A.R. FDCT 0074/2023/RIA2, and the
	University of Macau grants MYRG2020-00039-FST, MYRG-GRG2023-00088-FST.
	
	W. Zhang is supported by NNSFC grant 11931014.


	
	\bibliographystyle{siam}
	\normalem
	\bibliography{bibliography_beta_mixing.bib}
	
	
	

\end{document}